%% file: main.tex
\RequirePackage{fix-cm}
\documentclass[smallextended]{svjour3}
\smartqed
\usepackage{graphicx}
\usepackage{amsmath}
\usepackage{hyperref}
\usepackage{amsfonts}
\usepackage{bm}
\usepackage{tikz}
\usepackage{subcaption}
\usepackage{xcolor}
\usepackage{listings}
\usepackage[ruled]{algorithm2e}

\journalname{}

\input{defs.tex}

\begin{document}

\title{A Unified and Scalable Method for \\ Optimization over Graphs of Convex Sets
\thanks{The material presented in this paper is partially based on the author's PhD thesis~\cite{marcucci2024graphs}.}}
\titlerunning{Solving Optimization Problems over Graphs of Convex Sets}
\author{Tobia Marcucci}
\institute{T. Marcucci \at
University of California, Santa Barbara \\
Department of Electrical and Computer Engineering \\
\email{marcucci@ucsb.edu}
}

\date{\today}

\maketitle

\begin{abstract}
A Graph of Convex Sets (GCS) is a graph in which vertices are associated with convex programs and edges couple pairs of programs through additional convex costs and constraints.
Any optimization problem over an ordinary weighted graph (e.g., the shortest-path, the traveling-salesman, and the minimum-spanning-tree problems) can be naturally generalized to a GCS, yielding a new class of problems at the interface of combinatorial and convex optimization with numerous applications.
In this paper, we introduce a unified method for solving any such problem.
Starting from an integer linear program that models an optimization problem over a weighted graph, our method automatically produces an efficient mixed-integer convex formulation of the corresponding GCS problem.
This formulation is based on homogenization (perspective) transformations, and the resulting program is solved to global optimality using off-the-shelf branch-and-bound solvers.
We implement this framework in \texttt{GCSOPT}, an open-source and easy-to-use Python library designed for fast prototyping.
We illustrate the versatility and scalability of our approach through multiple numerical examples and comparisons.

\keywords{Combinatorial optimization \and Graph optimization \and Convex optimization \and Mixed-integer convex optimization \and Perspective formulation}
\subclass{90-08 \and 90-04 \and 90-10}
\end{abstract}

\newpage

\section{Introduction}
\label{sec:intro}

Many problems in combinatorial optimization and graph theory can be stated as follows:
\begin{multline}
\label{eq:word_statement}
\text{\emph{Given a weighted graph $G$ and a set $\cH$ of admissible subgraphs,}} \\
\text{\emph{find a minimum-weight subgraph $H \in \cH$.}}
\end{multline}
For example, in the Shortest-Path Problem (SPP), the set $\cH$ of admissible subgraphs consists of all paths between two fixed vertices in $G$.
In the Traveling-Salesman Problem (TSP), $\cH$ consists of all cycles that visit each vertex exactly once.
In the Minimum-Spanning-Tree Problem (MSTP), $\cH$ is the set of all trees that reach every vertex.
This paper addresses a generalization of problem~\eqref{eq:word_statement} where the weighted graph is replaced by a Graph of Convex Sets (GCS).

A GCS is a graph in which each vertex is paired with a convex program, defined by a continuous variable, a convex constraint set, and a convex objective function.
The edges of a GCS link pairs of these programs through additional convex costs and constraints.
Any discrete optimization problem of the form~\eqref{eq:word_statement} is naturally generalized to a GCS.
Indeed, if we fix the continuous variables of the GCS, we obtain a weighted graph over which we can solve problem~\eqref{eq:word_statement}.
The challenge in a GCS problem is to optimize the continuous variables so that the resulting discrete problem achieves the smallest optimal value.
This leads to a new class of problems at the interface of combinatorial and convex optimization with many applications.

The GCS framework has been recently introduced in~\cite{marcucci2024shortest} with a focus on the SPP, and has since been applied to a wide range of real-world robotics problems~\cite{marcucci2023motion,kurtz2023temporal,cohn2024non,graesdal2024towards,philip2024mixed,morozov2024multi,natarajan2024implicit,morozov2025mixed,wei2025hierarchical,lin2025towards}.
The first contribution of this paper is to extend the approach presented in~\cite{marcucci2024shortest} for solving the SPP in GCS to any GCS problem.
The input of our method is the formulation of problem~\eqref{eq:word_statement} as an Integer Linear Program (ILP).
These ILPs have been deeply studied, and are readily available in combinatorial-optimization textbooks~\cite{schrijver2003combinatorial,korte2018combinatorial,papadimitriou1998combinatorial,nemhauser1999integer,conforti2014integer}.
They typically have one binary variable for each vertex and edge in the graph.
The role of these variables is to take unit value if the corresponding vertex or edge is part of the subgraph $H$, and evaluate to zero otherwise.
Our method translates this ILP into an efficient Mixed-Integer Convex Program (MICP) that models the corresponding GCS problem.
This translation is based on homogenization transformations (also known as perspective formulations): a popular tool in mixed-integer programming~\cite{ceria1999convex,stubbs1999branch,frangioni2006perspective,gunluk2010perspective,moehle2015perspective} that allows us to activate and deactivate the convex costs and constraints in a GCS using the ILP binary variables.

Our MICPs are reliably solved with standard branch-and-bound solvers, which either return a globally optimal solution or certify infeasibility.
This process is facilitated by the small number of variables and constraints in our programs, and is especially effective when the original ILP is well approximated by its convex relaxation, since this property is typically inherited by our MICP.

Our second contribution is \texttt{GCSOPT}, an open-source Python library for formulating and solving GCS problems, designed for ease of use and fast prototyping.
\texttt{GCSOPT} provides a high-level interface for defining GCS vertices and connecting them with edges.
Convex sets and functions are specified using the syntax of \texttt{CVXPY}~\cite{diamond2016cvxpy}: a popular Python library for convex optimization.
Internally, these sets and functions are translated into conic form, enabling a straightforward computation of their homogenization.
The MICP is constructed automatically and shipped to state-of-the-art solvers (e.g., \texttt{Gurobi}, \texttt{Mosek}, or \texttt{CPLEX}), then the GCS variables are populated with their optimal value.
Thanks to this automation, users of \texttt{GCSOPT} do not need any significant expertise in mixed-integer programming: a basic familiarity with graphs and convex optimization is sufficient to model and solve complex GCS problems.

We demonstrate the broad applicability of our framework through a wide range of numerical examples.
We also show that our method consistently outperforms alternative approaches for solving GCS problems globally.

\subsection{Related works}

The GCS framework is closely related to several well-studied problems in graph theory and combinatorial optimization.
Here we review these connections, emphasizing similarities and key differences.

\subsubsection{Graph-structured convex optimization}

When the only admissible subgraph is the entire graph, $\cH = \{G\}$, a GCS problem reduces to a purely continuous convex program with graph structure.
Such programs are the focus of the software libraries \texttt{SnapVX}~\cite{hallac2017snapvx} and \texttt{Plasmo.jl}~\cite{jalving2022graph}.
The former is a solver based on the alternating-direction method of multipliers, and the latter is a modeling and solution framework.
In this paper, we add an extra layer of complexity by posing combinatorial questions on top of these graph-structured convex programs.

\subsubsection{Graph problems with neighborhoods}

Variants of classical graph optimization problems, in which vertices can be selected within prescribed continuous sets, are known in the literature as problems \emph{with neighborhoods}.
Examples are the TSP~\cite{arkin1994approximation,gudmundsson1999fast,dumitrescu2003approximation,de2005tsp,gentilini2013travelling,puerto2024hampered}, the MSTP~\cite{yang2007minimum,disser2014rectilinear,dorrigiv2015minimum,blanco2017minimum,blanco2025fixed}, the SPP~\cite{disser2014rectilinear}, the Facility-Location Problem (FLP)~\cite{blanco2019ordered,brimberg2002locating}, and matching problems~\cite{espejo2022minimum} with neighborhoods.
These problems have been approached either through approximation algorithms~\cite{arkin1994approximation,mata1995approximation,dumitrescu2003approximation,fourney2024mobile} or by formulating them as Mixed-Integer Nonconvex Programs (MINCP) that can be solved exactly~\cite{gentilini2013travelling,blanco2017minimum,fourney2024mobile}.
However, the former algorithms are typically limited to planar neighborhoods with simple shape (e.g., lines, circles, or rectangles), while the latter rely on spatial branch-and-bound algorithms, which have significantly improved in recent years, but still struggle with large-scale problems (see also the comparison in Section~\ref{sec:comparison}).

GCS problems generalize most common graph problems with neighborhoods, since our objective functions and constraints need only be convex.
As already shown by the many robotics applications of the SPP in GCS, our MICPs can scale to large graphs and high-dimensional spaces.
For instance, the motion-planning problems in~\cite{marcucci2023motion} involve up to $2{,}500$ vertices, $5{,}000$ edges, and vertex variables in $70$ dimensions.
Additionally, the method proposed in this paper is unified and can solve any GCS problem, as opposed to the works above that are tailored to a specific graph problem with neighborhoods.

\subsubsection{Generalized network-design and Steiner problems}
	
Generalized network-design problems~\cite{feremans2003generalized,pop2012generalized} and generalized Steiner problems~\cite{dror2000generalizedsteiner} can be viewed as discrete versions of graph problems with neighborhoods.
The vertex set is divided into clusters, and the constraints are enforced at the cluster level rather than on individual vertices.
For example, the generalized TSP has been studied in~\cite{noon1993efficient,fischetti1995symmetric}, the MSTP in~\cite{myung1995generalized,dror2000generalizedspanning,feremans2004generalized}, the SPP in~\cite{li1995shortest}, the vehicle routing problem in~\cite{ghiani2000efficient}, and graph coloring in~\cite{li2000partition,demange2015some}.
GCS problems admit a natural discretization as graph problems with clusters.
However, this approximation quickly becomes impractical in high dimensions, where dense sampling is infeasible, or in problems where equality constraints make the feasible set lower-dimensional and difficult to discretize.

\subsection{Outline}

This paper is organized as follows.
In Section~\ref{sec:statement}, we give a formal statement of the GCS problem.
Section~\ref{sec:ilp} contains background material on ILP formulations of discrete optimization problems over graphs.
Our MICP formulation is introduced in Section~\ref{sec:method} and further refined in Section~\ref{sec:lemma_analysis}.
Section~\ref{sec:gcs_problems} illustrates multiple examples of GCS problems along with their MICP formulations.
Section~\ref{sec:nonlinear} describes how some simplifying assumptions made in the previous sections can be relaxed.
All numerical examples and comparisons are presented in Section~\ref{sec:examples}, while the Python library \texttt{GCSOPT} is described in Section~\ref{sec:gcsopt}.
Finally, Section~\ref{sec:conclusions} presents concluding remarks and directions for future work.

\section{Problem statement}
\label{sec:statement}

Let $G=(\cV, \cE)$ be a weighted graph, either directed or undirected, with vertex set $\cV$ and edge set $\cE \subset \cV^2$.
Let $c_v \in \reals$ and $c_e \in \reals$ denote the weights of vertex $v \in \cV$ and edge $e \in \cE$, respectively.
More formally, problem~\eqref{eq:word_statement} can be restated as what we refer to as a \emph{graph optimization problem}:
\begin{subequations}
\label{eq:graph_problem}
\begin{align}
\minimize \quad & \sum_{v \in \cW} c_v + \sum_{e \in \cF} c_e \\
\label{eq:graph_problem_constraint}
\subjectto \quad & H = (\cW, \cF) \in \cH.
\end{align}
\end{subequations}
Here the variable is the subgraph $H$, which has vertex set $\cW \subseteq \cV$ and edge set $\cF \subseteq \cW^2 \cap \cE$.
As in~\eqref{eq:word_statement}, the set $\cH$ consists of the admissible subgraphs of $G$ (e.g., paths, tours, or spanning trees).
The objective function is equal to the total weight of $H$, i.e., the sum of its vertex and edge weights.

A GCS is a generalization of an ordinary weighted graph.
Each vertex $v \in \cV$ of a GCS is paired with a convex program, which is defined by
\begin{itemize}
\item
a variable $\bx_v \in \reals^{n_v}$ of dimension $n_v > 0$,
\item
a closed convex constraint set $\cX_v \subseteq \reals^{n_v}$,
\item
a convex objective function $f_v : \reals^{n_v} \rightarrow \reals$.
\end{itemize}
Each edge $e=[v,w] \in \cE$ in a GCS (we use square brackets for edges that may be directed or undirected) couples the convex programs of vertices $v$ and $w$ through
\begin{itemize}
\item
a closed convex constraint set $\cX_e \subseteq \reals^{n_v + n_w}$,
\item
a convex objective function $f_e : \reals^{n_v + n_w} \rightarrow \reals$.
\end{itemize}
Additional assumptions on these sets and functions will be made in different sections of the paper.
A simple sufficient condition under which all our results apply is that the sets $\cX_v$ are bounded for all $v \in \cV$.

The generalization of problem~\eqref{eq:graph_problem} over a GCS reads as follows:
\begin{subequations}
\label{eq:gcs_problem}
\begin{align}
\label{eq:gcs_problem_obj}
\minimize \quad & \sum_{v \in \cW} f_v(\bx_v) + \sum_{e = [v,w] \in \cF} f_e(\bx_v, \bx_w) \\
\subjectto \quad
\label{eq:gcs_subgraph}
& H = (\cW, \cF) \in \cH, \\
\label{eq:gcs_problem_vertex}
& \bx_v \in \cX_v, && v \in \cW, \\
\label{eq:gcs_problem_edge}
& (\bx_v, \bx_w) \in \cX_e, && e = [v,w] \in \cF.
\end{align}
\end{subequations}
Here the variables are both the discrete subgraph $H$ and the continuous vectors $\bx_v$ for $v \in \cW$.
The first constraint is inherited from the graph optimization problem.
The second and third enforce the vertex and edge constraints for the selected subgraph $H$.
The objective function is equal to the total cost of $H$.

\begin{example}
Figure~\ref{fig:demo} shows an SPP, a TSP, and an MSTP in GCS.
The first problem is posed over a directed graph $G$ with $|\cV| = 9$ vertices and $|\cE|=12$ edges.
The other two involve an undirected graph with $9$ vertices and $16$ edges.
For each problem, the set $\cH$ consists of the admissible subgraphs described at the beginning of Section~\ref{sec:intro}.
In the SPP, the path is required to start at the bottom-left vertex and end at the top-right vertex.
For all problems and vertices $v \in \cV$, the constraint sets $\cX_v \subset \reals^2$ are circles of radius $0.3$ centered on a grid with integer coordinates.
Each edge $e=[v,w] \in \cE$ has the objective function $f_e(\bx_v, \bx_w) = \|\bx_w - \bx_v\|_2$.
There are no vertex costs or edge constraints, i.e., $f_v(\bx_v) = 0$ and $\cX_e = \reals^4$ for all $v \in \cV$ and $e \in \cE$.

\begin{figure}
\centering
\begin{subfigure}{0.32\textwidth}
\includegraphics[width=\columnwidth]{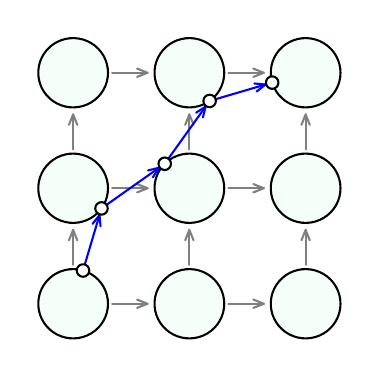}
\caption{SPP in GCS.}
\label{fig:demo_spp}
\end{subfigure}
\begin{subfigure}{0.32\textwidth}
\includegraphics[width=\columnwidth]{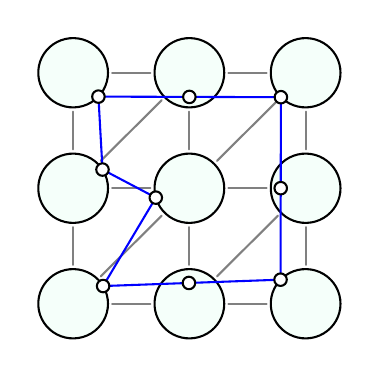}
\caption{TSP in GCS.}
\end{subfigure}
\begin{subfigure}{0.32\textwidth}
\includegraphics[width=\columnwidth]{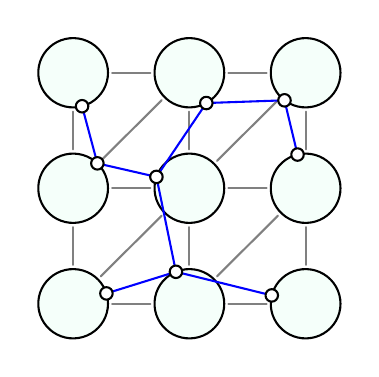}
\caption{MSTP in GCS.}
\label{fig:demo_mstp}
\end{subfigure}
\caption{
Examples of GCS problems.
The vertex constraint sets are circles and the edge costs penalize the Euclidean distance between the connected points.
}
\label{fig:demo}
\end{figure}
\end{example}

Figure~\ref{fig:demo} shows the optimal subgraphs $H=(\cW, \cF)$ for the three problems.
The optimal variables $\bx_v$ for $v \in \cW$ are depicted as white dots.
The selected edges $e \in \cF$ connect the corresponding white dots and are shown in blue.

\section{Graph optimization problems as integer linear programs}
\label{sec:ilp}

Integer programming offers a unified framework for solving graph optimization problems: once problem~\eqref{eq:graph_problem} is cast as an ILP, it can be reliably solved to global optimality with highly effective branch-and-bound solvers.
In this section, we examine some important properties of these ILPs that will play a central role in formulating the GCS problem~\eqref{eq:gcs_problem} as an MICP.

The first step in formulating problem~\eqref{eq:graph_problem} as an ILP is to parameterize the subgraph $H = (\cW, \cF)$ through its incidence (or characteristic) vector
$$
\by^H \in \{0,1\}^{\cV \cup \cE}.
$$
The entries of this vector are indexed by the elements of the set $\cV \cup \cE$, and have value
$$
y_v^H = \begin{cases}
1 & \text{if } v \in \cW \\
0 & \text{if } v \notin \cW
\end{cases}, \qquad
y_e^H = \begin{cases}
1 & \text{if } e \in \cF \\
0 & \text{if } e \notin \cF
\end{cases},
$$
for all $v \in \cV$ and $e \in \cE$.
Let
$$
\cY_\cH = \{\by^H : H \in \cH \}
$$
denote the set of incidence vectors that represent an admissible subgraph.
We enclose this set with a polytope $\cY \subseteq \reals^{\cV \cup \cE}$ whose binary-valued elements are exactly the admissible incidence vectors:
\begin{align}
\label{eq:incidence}
\cY \cap \{0,1\}^{\cV \cup \cE} =
\cY_\cH.
\end{align}
Using the shorthand notation $\cY\bin = \cY \cap \{0,1\}^{\cV \cup \cE}$ for the binary elements of $\cY$, problem~\eqref{eq:graph_problem} can be rewritten as the ILP
\begin{subequations}
\label{eq:graph_ilp}
\begin{align}
\label{eq:graph_ilp_obj}
\minimize \quad & \sum_{v \in \cV} c_v y_v + \sum_{e \in \cE} c_e y_e \\
\label{eq:graph_ilp_constr}
\subjectto \quad & \by \in \cY\bin.
\end{align}
\end{subequations}
The decision variable is the incidence vector $\by$, with the superscript $H$ omitted for simplicity.
The binary variables in the objective select the weights of the vertices and edges included in the subgraph.
This ILP has the same optimal value as problem~\eqref{eq:graph_problem}, and its optimal solutions $\by$ are the incidence vectors of the optimal subgraphs $H$.
The ILP convex relaxation is obtained simply by replacing the polytope $\cY\bin$ with $\cY$ in~\eqref{eq:graph_ilp_constr}.

\subsection{Compact versus strong formulations}

How do we choose the polytope $\cY$ in practice?
For any family $\cH$ of subgraphs, there are infinitely many polytopes $\cY$ satisfying the incidence condition~\eqref{eq:incidence}.
Among these, we seek one that has few facets and for which the inclusion
\begin{align}
\label{eq:perfect_formulation}
\conv (\cY_\cH) \subseteq \cY
\end{align}
is sufficiently tight, where $\conv$ denotes the convex hull.
A polytope $\cY$ with few facets yields a computationally light ILP~\eqref{eq:graph_ilp}, which we refer to as a \emph{compact} formulation.
While, when the inclusion~\eqref{eq:perfect_formulation} is tight, the ILP is well approximated by its convex relaxation and the branch-and-bound process converges quickly.
In this case, the formulation is said to be \emph{strong}.  
Furthermore, if the inclusion~\eqref{eq:perfect_formulation} holds with equality (i.e., is perfectly tight), then the ILP can be solved exactly through its convex relaxation, and the formulation is said to be \emph{perfect}.

Designing an ILP that is both compact and strong is sometimes impossible, as tightening the inclusion~\eqref{eq:perfect_formulation} might require adding many facets to $\cY$.
This trade-off is well understood for most graph optimization problems, and efficient ILP formulations can be found in standard textbooks~\cite{schrijver2003combinatorial,korte2018combinatorial,papadimitriou1998combinatorial,nemhauser1999integer,conforti2014integer} (see also Section~\ref{sec:gcs_problems} for a variety of examples).

\subsection{Nonnegative weights}

When the weights $c_v$ for $v \in \cV$ and $c_e$ for $e \in \cE$ are nonnegative, the incidence condition~\eqref{eq:incidence} can be relaxed to
\begin{align}
\label{eq:incidence_dominant}
\cY\bin + \nnintegers^{\cV \cup \cE} =
\cY_\cH + \nnintegers^{\cV \cup \cE},
\end{align}
where $\nnintegers$ is the set of nonnegative integers and  the plus symbols denote Minkowski sums.
In words, we ignore any difference that the sets $\cY\bin$ and $\cY_\cH$ might have in the positive directions, and ask only that the two sets have equal minimal elements (with respect to the partial order $\leq$).
Indeed, if the ILP~\eqref{eq:graph_ilp} is feasible, one of these minimal elements is necessarily optimal.

Similarly, when the weights are nonnegative, and we analyze the strength of an ILP formulation, we compare the \emph{dominants} of the two sets in~\eqref{eq:perfect_formulation}:
\begin{align}
\label{eq:perfect_formulation_dominant}
\conv(\cY_\cH) + \nnreals^{\cV \cup \cE} \subseteq
\cY + \nnreals^{\cV \cup \cE},
\end{align}
where $\nnreals$ are the nonnegative reals.
If this inclusion is tight or holds with equality, the ILP formulation is again said to be \emph{strong} or \emph{perfect}, respectively.

\subsection{Subgraph polytope}

Although the polytope $\cY$ is problem dependent, the incidence vector $\by$ must always represent a subgraph $H$ of $G$.
Hence, regardless of the signs of the weights, we can assume that $\cY$ is contained in the \emph{subgraph polytope}~\cite{conforti2015subgraph}:
\begin{align}
\label{eq:subgraph_polytope}
\cY \subseteq
\cY\sub =
\{\by \in [0,1]^{\cV \cup \cE} : y_v \geq y_e \tforall v \in \cV \tand e \in \incident{v}\},
\end{align}
where $\incident{v}$ is the set of edges incident with vertex $v$.
The inequality $y_v \geq y_e$ ensures that if a vertex $v$ is excluded from the subgraph $H$, so are its incident edges $e$.
Equivalently, if an edge $e =[v,w]$ is included in the subgraph $H$, so are the vertices $v$ and $w$.
The following bilinear equality encodes the same logical implication and will be used multiple times:
\begin{align}
\label{eq:binary_product}
& y_v y_e = y_e, && v \in \cV, \ e \in \incident{v}.
\end{align}
If $y_v = 0$ then $y_e = 0$ and if $y_e = 1$ then $y_v = 1$.

\section{Base mixed-integer convex formulation of the GCS problem}
\label{sec:method}

This section presents an initial MICP formulation that applies to any GCS problem~\eqref{eq:gcs_problem}.
Our construction extends the one proposed in~\cite{marcucci2024shortest} for the SPP in GCS.
First, we transform the ILP~\eqref{eq:graph_ilp} into an MINCP that models the corresponding GCS problem, but is impractical to solve due to the nonconvexity of its constraints.
Second, we design a convex relaxation that yields an equivalent, but easier to solve, MICP.
In Section~\ref{sec:lemma_analysis}, we will show how this base MICP can be automatically tailored to specific classes of GCS problems to improve its computational efficiency.

We will make two main simplifying assumptions.
\begin{assumption}[Bounded sets]
\label{ass:bounded}
For all vertices $v \in \cV$ and edges $e \in \cE$, the sets $\cX_v$ and $\cX_e$ are bounded.
\end{assumption}

\begin{assumption}[Linear objectives]
\label{ass:linear}
For all vertices $v \in \cV$ and edges $e = [v,w] \in \cE$, the functions $f_v$ and $f_e$ are linear, i.e., there exist vectors $\bc_v \in \reals^{n_v}$ and $\bc_e \in \reals^{n_v + n_w}$ such that $f_v(\bx_v) = \bc_v^\top \bx_v$ and $f_e(\bx_v, \bx_w) = \bc_e^\top (\bx_v, \bx_w)$.
\end{assumption}
While not essential, these assumptions simplify the exposition and avoid multiple technical subtleties.
In Section~\ref{sec:nonlinear}, we will show how our MICP formulation can be extended to unbounded sets and nonlinear objective functions.

\subsection{Homogenization}
\label{sec:homog}

The following operation is at the core of our method.
It allows us to switch on and off convex constraints using binary variables.

\begin{definition}[Homogenization]
\label{def:homog}
The \emph{homogenization} of a compact convex set $\cX \in \reals^n$ is the set
\begin{align*}
\tilde \cX = \{(\bx, y) \in \reals^{n+1}: y \geq 0, \ \bx \in y \cX \}.
\end{align*}
\end{definition}

\begin{example}
The homogenization of the closed interval $\cX = [1,2] \subset \reals$ is the set $\tilde \cX = \{(x, y) : y \leq x \leq 2y \}$ shown in Figure~\ref{fig:homog_bounded}.
\end{example}

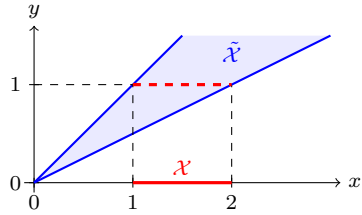
\begin{figure}
\centering
\begin{tikzpicture}[scale=1.3]
\draw[->] (-0.1,0) -- (3.1,0) node[right] {$x$};
\draw[->] (0,-0.1) -- (0,1.6) node[above] {$y$};
\foreach \x in {0, 1, 2} {\draw (\x,0.05) -- (\x,-0.05) node[below] {\small $\x$};}
\foreach \y in {0, 1} {\draw (0.05,\y) -- (-0.05,\y) node[left] {\small $\y$};}
\fill[blue!20,opacity=0.4] (0,0) -- (3,1.5) -- (1.5,1.5) -- cycle;
\draw[thick,blue] (0,0) -- (1.5,1.5);
\draw[thick,blue] (0,0) -- (3,1.5);
\draw[dashed]  (0,1) -- (1,1);
\draw[dashed]  (1,0) -- (1,1);
\draw[dashed]  (2,0) -- (2,1);
\draw[very thick,red] (1,0) -- (2,0);
\draw[very thick,red,dashed] (1,1) -- (2,1);
\draw (1.5,0) node[above] {\color{red} $\cX$};
\draw (2,1.333333333) node {\color{blue} $\tilde \cX$};
\end{tikzpicture}
\caption{Homogenization $\tilde \cX$ of the closed interval $\cX = [1, 2]$.}
\label{fig:homog_bounded}
\end{figure}

The homogenization $\tilde \cX$ is easily verified to be a closed convex cone.
In fact, it could be equivalently defined as the cone generated by the points $(\bx, 1)$ with $\bx \in \cX$.
Especially important for us are the observations that $(\bx, 1) \in \tilde \cX$ is equivalent to $\bx \in \cX$ and $(\bx, 0) \in \tilde \cX$ is equivalent to $\bx = \bzero$ (since $0 \cX = \{\bzero\}$).

In Section~\ref{sec:homog_unbd}, we will extend Definition~\ref{def:homog} to unbounded sets, and, in Section~\ref{sec:gcsopt_micp}, we will show how homogenization transformations are well suited for numerical computations and amenable to standard conic-optimization solvers.

\begin{remark}
The term homogenization is not fully standard.
It is used, for example, in~\cite[Definition~1.13]{ziegler2012lectures}.
The same operation is called \emph{conic hull} in~\cite[Section~3.3]{ben2001lectures}, while it is commonly named \emph{perspective} when applied to functions (see, e.g.,~\cite[Section~IV.2.2]{hiriart2013convex} or~\cite[Section~2.3.3]{boyd2004convex}).
The same construction appears frequently in~\cite[Section~8]{rockafellar1970convex}, but is used without an explicit name.
\end{remark}

\subsection{Mixed-integer nonconvex formulation}
\label{sec:mincp}

A natural extension of the ILP~\eqref{eq:graph_ilp} to model the GCS problem~\eqref{eq:gcs_problem} leads to the following program:
\begin{subequations}
\label{eq:mincp_initial}
\begin{align}
\label{eq:mincp_initial_obj}
\minimize \quad & \sum_{v \in \cV} y_v \bc_v^\top \bx_v + \sum_{e = [v,w] \in \cE} y_e \bc_e^\top (\bx_v, \bx_w)  \\
\subjectto \quad
& \by \in \cY\bin, \\
\label{eq:mincp_initial_v}
& y_v (\bx_v, 1) \in \tilde \cX_v, && v \in \cV, \\
\label{eq:mincp_initial_e} 
& y_e (\bx_v, \bx_w, 1) \in \tilde \cX_e, && e = [v,w] \in \cE.
\end{align}
\end{subequations}
Here the decision variables are both the binary variables $\by$ and the continuous variables $\bx_v$ for $v \in \cV$.
The objective uses the binary variables to activate and deactivate the cost contributions of the GCS vertices and edges, which are linear functions by Assumption~\ref{ass:linear}.
The first constraint comes from the ILP, and forces $\by$ to be the incidence vector of an admissible subgraph.
The second constraint switches the vertex constraints on and off: when $y_v = 1$ it gives us $(\bx_v, 1) \in \tilde \cX_v$, which is equivalent to $\bx_v \in \cX_v$, and when $y_v = 0$ it gives us $(\bzero, 0) \in \tilde \cX_v$, which is always satisfied since $\tilde \cX_v$ is a closed cone.
The third constraint plays the same role for the edge constraints.

The equivalence of problem~\eqref{eq:mincp_initial} and the GCS problem~\eqref{eq:gcs_problem} is easily established, provided that the polytope $\cY$ satisfies the incidence condition~\eqref{eq:incidence} or, when the objective functions $f_v$ for $v \in \cV$ and $f_e$ for $e \in \cE$ are nonnegative, the relaxed incidence condition~\eqref{eq:incidence_dominant}.

We rewrite problem~\eqref{eq:mincp_initial} in an equivalent form to simplify the upcoming derivations.
First, observe that the objective~\eqref{eq:mincp_initial_obj} as well as the constraints~\eqref{eq:mincp_initial_v} and~\eqref{eq:mincp_initial_e} are nonconvex due to products between binary and continuous variables.
However, they all become convex if expressed in terms of the following auxiliary variables:
\begin{subequations}
\label{eq:prod_vars}
\begin{align}
\label{eq:prod_vars_v}
& \bz_v = y_v \bx_v, && v \in \cV, \\
\label{eq:prod_vars_ve}
& \bz_v^e = y_e \bx_v, \ \bz_w^e = y_e \bx_w, && e = [v, w] \in \cE.
\end{align}
\end{subequations}
Secondly, we express the auxiliary variables $\bz_v^e$ just defined in terms of $\bz_v$ rather than $\bx_v$:
\begin{align}
\label{eq:prod_vars_ve_reindex}
& \bz_v^e = y_e \bx_v = y_e y_v \bx_v = y_e \bz_v, && v \in \cV, \ e \in \incident{v}.
\end{align}
The first equality is equivalent to~\eqref{eq:prod_vars_ve} after a reindexing.
The second and the third equalities follow from~\eqref{eq:binary_product} and~\eqref{eq:prod_vars_v}, respectively.

Collecting all the modifications, problem~\eqref{eq:mincp_initial} is reformulated as
\begin{subequations}
\label{eq:mincp}
\begin{align}
\minimize \quad &
\sum_{v \in \cV} \bc_v^\top \bz_v + \sum_{e = [v,w] \in \cE}
\bc_e^\top (\bz_v^e, \bz_w^e) \\
\subjectto \quad
\label{eq:mincp_int}
& \by \in \cY\bin, \\
\label{eq:mincp_v}
& (\bz_v, y_v) \in \tilde \cX_v, && v \in \cV, \\
\label{eq:mincp_e}
& (\bz_v^e, \bz_w^e, y_e) \in \tilde \cX_e, && e = [v,w] \in \cE, \\
\label{eq:gcs_bilin}
& \bz_v^e = y_e \bz_v, && v \in \cV,\ e \in \incident{v}.
\end{align}
\end{subequations}
The objective, the second constraint, and the third constraint are the result of substituting the new variables~\eqref{eq:prod_vars} in the original problem.
The last constraint is taken from~\eqref{eq:prod_vars_ve_reindex}, and is equivalent to~\eqref{eq:prod_vars_ve}.
We omitted the constraint $\bz_v = y_v \bx_v$ from~\eqref{eq:prod_vars_v} since the variables $\bx_v$ would appear only in this constraint, and the fact that $\bz_v = \bzero$ when $y_v=0$ is already implied by~\eqref{eq:mincp_v}.
Therefore, we can solve the problem without constraint~\eqref{eq:prod_vars_v} and, afterwards, define $\bx_v = \bz_v$ if $y_v =1$ or leave $\bx_v$ undefined if $y_v=0$, for all $v \in \cV$.

Constraint~\eqref{eq:gcs_bilin} is nonconvex and makes problem~\eqref{eq:mincp} an MINCP.
This MINCP is akin to existing ones for graph problems with neighborhoods~\cite{gentilini2013travelling,blanco2017minimum,fourney2024mobile}, but applies to any GCS problem, uses homogenization transformations, and includes edge constraints.
Although MINCP solvers have improved substantially in recent years, their scalability remains limited: our next step is to transform problem~\eqref{eq:mincp} into a more tractable MICP by leveraging the bilinear structure of constraint~\eqref{eq:gcs_bilin}.

\subsection{Convex relaxation of the bilinear constraint}

We obtain our MICP by replacing the bilinear constraint~\eqref{eq:gcs_bilin} with a family of convex constraints.
Ideally, these convex constraints should:
\begin{itemize}
\item
Be equivalent to the bilinear constraint when the vector $\by$ has binary entries, resulting in a \emph{correct} MICP formulation of the GCS problem~\eqref{eq:gcs_problem}.
\item
Envelop the bilinear constraint tightly when the vector $\by$ has fractional entries, resulting in a \emph{strong} MICP formulation.
\item
Be few in number, resulting in a \emph{compact} MICP formulation.
\end{itemize}
Below, we present the approach that, in our experience, best balances these competing objectives.
Section~\ref{sec:comparison} reports a numerical comparison with a simpler formulation based on McCormick envelopes~\cite{mccormick1976computability}.

The next lemma parallels~\cite[Lemma~5.4]{marcucci2024shortest}, and gives us an algorithmic way of enveloping the bilinear equality~\eqref{eq:gcs_bilin} with convex constraints.
Recall that a constraint is \emph{valid} for a set if it is satisfied by all points in the set, and \emph{valid} for an optimization problem if it is valid for the problem's feasible set.

\begin{lemma}[Valid-constraint generation]
\label{lem:valid_ineq}
For some vertex $v \in \cV$, assume that the following linear inequality is valid for the set $\cY\bin$:
\begin{align}
\label{eq:valid_ineq}
a y_v + \sum_{e \in \incident{v}} b_e y_e \geq 0.
\end{align}
Then the following convex constraint is valid for the MINCP~\eqref{eq:mincp}:
\begin{align}
\label{eq:valid_ineq_lift}
a (\bz_v, y_v) + \sum_{e \in \incident{v}} b_e (\bz_v^e, y_e) \in \tilde \cX_v.
\end{align}
\end{lemma}

\begin{proof}
We multiply the constraint $(\bz_v, y_v) \in \tilde \cX_v$ from~\eqref{eq:mincp_v} by the left-hand side of~\eqref{eq:valid_ineq}, yielding the valid constraint
$a (y_v \bz_v, y_v^2) + \sum_{e \in \incident{v}} b_e (y_e \bz_v, y_e y_v) \in \tilde \cX_v$.
This constraint is convexified through the identities $y_v \bz_v = \bz_v$ and $y_v^2 = y_v$ for all $v \in \cV$, together with the bilinear equalities~\eqref{eq:gcs_bilin} and~\eqref{eq:binary_product}.
\hfill\qed
\end{proof}
Lemma~\ref{lem:valid_ineq} allows us to translate any linear inequality from the original ILP into a valid convex constraint for the MINCP, provided that the inequality involves only the binary variables associated with a vertex and its incident edges.

We illustrate the usage of Lemma~\ref{lem:valid_ineq} through a simple example.
Since the polytope $\cY$ is contained in the subgraph polytope in~\eqref{eq:subgraph_polytope}, for any vertex $v \in \cV$ and edge $e \in \incident{v}$, the inequalities $0 \leq y_e \leq y_v \leq 1$ are valid for $\cY\bin$.
The first two of these inequalities are linear and, using our lemma, we have
\begin{align*}
y_e \geq 0 & \quad \Longrightarrow \quad (\bz_v^e, y_e) \in \tilde \cX_v, && v \in \cV,\ e \in \incident{v}, \\
y_v- y_e \geq 0 & \quad \Longrightarrow \quad (\bz_v - \bz_v^e, y_v - y_e) \in \tilde \cX_v, && v \in \cV,\ e \in \incident{v}.
\end{align*}
The following theorem shows that these two convex constraints alone are sufficient to meet the first criterion above, i.e., they are equivalent to the bilinear constraint~\eqref{eq:gcs_bilin} when $\by$ is binary, yielding a correct MICP formulation.

\begin{theorem}[MICP correctness]
\label{th:micp}
The following MICP has the same optimal value as the GCS problem~\eqref{eq:gcs_problem}:
\begin{subequations}
\label{eq:micp}
\begin{align}
\label{eq:micp_obj}
\minimize \quad &
\sum_{v \in \cV} \bc_v^\top \bz_v + \sum_{e = [v,w] \in \cE} \bc_e^\top (\bz_v^e, \bz_w^e) \\
\subjectto \quad
\label{eq:micp_Y}
& \by \in \cY\bin, \\
\label{eq:micp_e0}
& (\bz_v^e, y_e) \in \tilde \cX_v, && v \in \cV,\ e \in \incident{v}, \\
\label{eq:micp_e1}
& (\bz_v - \bz_v^e, y_v - y_e) \in \tilde \cX_v, && v \in \cV,\ e \in \incident{v}, \\
\label{eq:micp_e_hom}
& (\bz_v^e, \bz_w^e, y_e) \in \tilde \cX_e, && e = [v,w] \in \cE.
\end{align}
\end{subequations}
A subgraph $H$ is optimal for problem~\eqref{eq:gcs_problem} if and only if its incidence vector $\by$ is optimal for this MICP.
Optimal values of the continuous variables $\bx_v$ associated with the subgraph $H$ are given by the optimal values of the MICP variables $\bz_v$ for all $v \in \cV$ such that $y_v = 1$.
\end{theorem}

The MICP~\eqref{eq:micp} involves the same variables as the original MINCP, and has a similar number of constraints.
Its convex relaxation is obtained by replacing $\cY^\text{bin}$ with $\cY$ in~\eqref{eq:micp_Y}.
Observe also that in the MICP we have omitted the constraint $(\bz_v, y_v) \in \tilde \cX_v$ from~\eqref{eq:mincp_v}: this is redundant since it is recovered by summing~\eqref{eq:micp_e0} and~\eqref{eq:micp_e1}, and using the fact that $\tilde \cX_v$ is a cone.

\begin{proof}[Theorem~\ref{th:micp}]
Since the MICP is a relaxation of the MINCP~\eqref{eq:mincp}, it suffices to show that the valid convex constraints imply the bilinear equality $\bz_v^e = y_e \bz_v$ that they replaced.
When $y_e=0$, constraint~\eqref{eq:micp_e0} gives us $(\bz_v^e, 0) \in \tilde \cX_v$, i.e., $\bz_v^e = \bzero$.
When $y_e=1$, we have $y_v=1$ if edge $e$ is incident with vertex $v$, and constraint~\eqref{eq:micp_e1} becomes $(\bz_v - \bz_v^e, 0) \in \tilde \cX_v$, i.e., $\bz_v = \bz_v^e$.
\hfill\qed
\end{proof}

Regarding the other two criteria above, the MICP formulation~\eqref{eq:micp} is compact but may still be weak.
To strengthen it, we must apply Lemma~\ref{lem:valid_ineq} to all suitable inequalities that are valid for $\cY\bin$ and specific to the underlying graph optimization problem.
This process is discussed further in the next section, and several examples are illustrated in Section~\ref{sec:gcs_problems}.

\section{Automatic tailoring of the mixed-integer convex program}
\label{sec:lemma_analysis}

We now describe how the base MICP~\eqref{eq:micp} can be automatically tailored to specific classes of GCS problems.
First, we discuss a few variations and extensions of our constraint-generation procedure.
Second, we present the algorithm that performs the automatic tailoring.
Finally, we provide a qualitative comparison of our method with alternative relaxation techniques.

\subsection{Variations and extensions of the constraint-generation procedure}
\label{sec:variants}

Lemma~\ref{lem:valid_ineq} is easily specialized to equality constraints.

\begin{corollary}[Constraint generation for equalities]
\label{cor:valid_eq}
For some vertex $v \in \cV$, assume that the linear equality $a y_v + \sum_{e \in \incident{v}} b_e y_e= 0$ is valid for the set $\cY\bin$.
Then the linear equality $a \bz_v + \sum_{e \in \incident{v}} b_e \bz_v^e = \bzero$ is valid for the MINCP~\eqref{eq:mincp}.
\end{corollary}

\begin{proof}
We multiply the scalar linear equality by $\bz_v$, then linearize the variable products as in the proof of Lemma~\ref{lem:valid_ineq}.
\hfill\qed
\end{proof}

Informally, the next proposition states that our procedure preserves constraint redundancy: given a set of linear inequalities amenable to Lemma~\ref{lem:valid_ineq}, it is sufficient to consider only those that are not implied by the others.

\begin{proposition}[Redundancy preservation]
\label{prop:redundancy}
Assume that the linear inequality~\eqref{eq:valid_ineq} is implied by $a_i y_v + \sum_{e \in \incident{v}} b_{i,e} y_e \geq 0$ for $i \in \cI$.
Then the convex constraint~\eqref{eq:valid_ineq_lift} is implied by $a_i (\bz_v, y_v) + \sum_{e \in \incident{v}} b_{i,e} (\bz_v^e, y_e) \in \tilde \cX_v$ for $i \in \cI$.
\end{proposition}

\begin{proof}
The assumption is equivalent to the existence of nonnegative coefficients $\lambda_i$ for $i \in \cI$ such that $(a, b_e) = \sum_{i \in \cI} \lambda_i (a_i, b_{i,e})$.
By combining the convex constraints with the same coefficients we get~\eqref{eq:valid_ineq_lift}.
\hfill\qed
\end{proof}
As an illustration of this proposition, recall that in the MICP~\eqref{eq:micp} we have omitted the constraint $(\bz_v, y_v) \in \tilde \cX_v$ since it is implied by $(\bz_v^e, y_e) \in \tilde \cX_v$ and $(\bz_v - \bz_v^e, y_v - y_e) \in \tilde \cX_v$.
Using Proposition~\ref{prop:redundancy}, we arrive to the same conclusion just by noticing that $y_e \geq 0$ and $y_v - y_e \geq 0$ imply $y_v \geq 0$.

The next lemma extends the applicability of our procedure to affine inequalities.
It shows that any affine inequality (involving the binary variables related to a common vertex $v \in \cV$) can be replaced by a linear inequality amenable to Lemma~\ref{lem:valid_ineq}, together with either the condition $y_v \leq 1$ or $y_v = 1$.

\begin{lemma}[From affine to linear inequalities]
\label{lem:linear_vs_affine}
For some vertex $v \in \cV$, assume that the affine inequality $a y_v + \sum_{e \in \incident{v}} b_e y_e + c \geq 0$ is valid for $\cY\bin$.
Consider the linear inequality $(a + c) y_v + \sum_{e \in \incident{v}} b_e y_e \geq 0$.
If $c \geq 0$ (respectively, $c < 0$), then the linear inequality and $y_v \leq 1$ (respectively, $y_v = 1$) are valid for $\cY\bin$ and imply the affine inequality.
\end{lemma}

\begin{proof}
For the validity of the linear inequality, note that this is obtained by multiplying the affine inequality by $y_v \geq 0$ and linearizing the variable products.
For the validity of $y_v=1$ when $c < 0$, note that setting $y_v=0$ would imply $y_e = 0$ for all $e \in \incident{v}$, and violate the affine inequality if $c < 0$.
The rest of the statement is easily shown.
\hfill\qed
\end{proof}
The last lemma is used as follows.
Given a candidate affine inequality, we replace it with the corresponding linear inequality and apply Lemma~\ref{lem:valid_ineq} to the latter.
If $c < 0$, we also add the constraint $y_v=1$ to our problem (if not already satisfied by the polytope $\cY$).

The next corollary adapts Lemma~\ref{lem:linear_vs_affine} to equality constraints.

\begin{corollary}[From affine to linear equalities]
\label{cor:linear_vs_affine_eq}
For some vertex $v \in \cV$, assume that the affine equality $a y_v + \sum_{e \in \incident{v}} b_e y_e + c = 0$ is valid for $\cY\bin$ and $c \neq 0$.
Then the linear equalities $(a + c) y_v + \sum_{e \in \incident{v}} b_e y_e = 0$ and $y_v=1$ are valid for $\cY\bin$ and imply the affine equality.
\end{corollary}

\subsection{Tailoring algorithm}
\label{sec:usage}

The base MICP~\eqref{eq:micp} is constructed leveraging only the linear inequalities that are valid for the subgraph polytope.
However, knowing the specific class of GCS problems and the corresponding polytope $\cY$, allows us to generate additional convex constraints and strengthen the MICP.

Ideally, to fully exploit the structure of a given polytope $\cY$ and maximize the MICP strength, we would proceed as follows.
For each vertex $v \in \cV$, let $\cY_v\bin$ denote the orthogonal projection of $\cY\bin$ onto the subspace of $y_v$ and $y_e$ for $e \in \delta_v$.
Find a minimal set of linear constraints that are valid for $\cY_v\bin$ and imply all other valid linear constraints (i.e., the extreme rays of the dual cone of $\cY_v\bin$).
Replace~\eqref{eq:micp_e0} and~\eqref{eq:micp_e1} with the convex constraints derived from these linear constraints using Lemma~\ref{lem:valid_ineq} (for inequalities) or Corollary~\ref{cor:valid_eq} (for equalities).
Since $\cY \subseteq \cY\sub$, the linear constraints imply the subgraph inequalities $y_v \ge y_e \ge 0$ for $e \in \incident{v}$, and the derived convex constraints imply~\eqref{eq:micp_e0} and~\eqref{eq:micp_e1} by Proposition~\ref{prop:redundancy}.
Hence, the resulting MICP is correct by Theorem~\ref{th:micp}.

In practice, however, computing the projection $\cY_v\bin$ is generally intractable.
Therefore, we construct our MICP in a more conservative way.
For all $v \in \cV$, we select the constraints that define $\cY$ and involve only the variables $y_v$ and $y_e$ for $e \in \delta_v$.
If a selected constraint is linear, we apply Lemma~\ref{lem:valid_ineq} or Corollary~\ref{cor:valid_eq}, and add the resulting convex constraint to the MICP~\eqref{eq:micp}.
If a constraint is affine, we linearize it using Lemma~\ref{lem:linear_vs_affine} (for inequalities) or Corollary~\ref{cor:linear_vs_affine_eq} (for equalities), add the constraint $y_v=1$ to the MICP when appropriate, and proceed as in the linear case.
Finally, if the linear and linearized constraints associated with vertex $v$ imply a subgraph inequality $y_e \geq 0$ or $y_v \ge y_e$ for some $e \in \incident{v}$, we remove the corresponding constraint from~\eqref{eq:micp_e0} or~\eqref{eq:micp_e1} since it is redundant.
The resulting MICP is again correct by Theorem~\ref{th:micp}.

Algorithm~\ref{alg:constraint_generation} summarizes the steps just described.
In the algorithm, $\cL_v$ represents the set of linear and linearized constraints associated with vertex $v$.

\begin{algorithm}[t]
\SetAlgoNoEnd
\DontPrintSemicolon
\caption{Tailoring of base MICP to a specific GCS problem}
\label{alg:constraint_generation}
\KwIns{GCS and polytope $\cY$}
\KwOut{tailored MICP}
initialize MICP as in~\eqref{eq:micp}\;
\ForEach{vertex $v \in \cV$}{
initialize empty set of linear constraints $\cL_v$\;
\ForEach{constraint defining $\cY$}{
\If{constraint involves only variables $y_v$ and $y_e$ for $e \in \delta_v$}{
\If{constraint is affine}{
linearize constraint using Lemma~\ref{lem:linear_vs_affine} or Corollary~\ref{cor:linear_vs_affine_eq}\;
add constraint $y_v=1$ to MICP if appropriate\;
}
add linear constraint to set $\cL_v$\;
add convex constraint generated by Lemma~\ref{lem:valid_ineq} or Corollary~\ref{cor:valid_eq} to MICP\;
}
}
\ForEach{edge $e \in \incident{v}$}{
\If{constraints in $\cL_v$ imply $y_e \ge 0$ or $y_v \ge y_e$}{
remove constraint generated by implied inequality from~\eqref{eq:micp_e0} or~\eqref{eq:micp_e1}\;
}
}
}
\end{algorithm}

\subsection{Alternative relaxation techniques}

The steps leading to our MICP generalize those presented in~\cite[Section~5]{marcucci2024shortest} for the SPP in GCS, with key improvements that exploit the structure of the subgraph polytope.
Specifically, our constraint-generation procedure focuses on linear constraints, in contrast to~\cite[Lemma~5.4]{marcucci2024shortest} which applies to affine constraints of the form considered in Lemma~\ref{lem:linear_vs_affine}.
While our procedure can easily be extended to affine constraints (see~\cite[Lemma~5.1]{marcucci2024graphs}), this is unnecessary since any affine constraint other than $y_v \leq 1$ or $y_v = 1$ reduces to a linear constraint that is at least as tight (see Lemma~\ref{lem:linear_vs_affine}).
In addition, the valid convex constraints generated from $y_v \leq 1$ and $y_v = 1$ via~\cite[Lemma~5.1]{marcucci2024graphs} are easily seen to be redundant for our MICP.
This observation streamlines the formulation of our MICPs by preventing the generation of redundant convex constraints, which are manually identified and eliminated in~\cite[Section~5.3]{marcucci2024shortest}.

Our procedure can also be extended to constraints involving all binary variables, and not just those related to a common vertex.
However, this requires introducing new variables that represent all possible products between a binary and a continuous variable.
In our experience, the resulting MICPs are stronger but much slower to solve than ours due to their larger size.

Finally, the core idea behind Lemma~\ref{lem:valid_ineq} can be extended to a wide class of bilinear constraints, as shown in~\cite[Section~7]{marcucci2024shortest}.
The same reference highlights the close connections between our technique and the Lov\'asz-Schrijver hierarchy~\cite{lovasz1991cones}, as well as other classical relaxation hierarchies~\cite{sherali1990hierarchy,parrilo2000structured,parrilo2003semidefinite,lasserre2001global}.

\section{Some GCS problems and their mixed-integer convex formulations}
\label{sec:gcs_problems}

This section presents several examples of GCS problems together with their MICP formulations.
For each problem, we recall the ILP formulation of the associated graph optimization problem and apply Algorithm~\ref{alg:constraint_generation} to tailor the base MICP~\eqref{eq:micp}.

The problems below admit several variants (e.g., directed versus undirected graphs or nonnegative versus sign-indefinite weights) and multiple ILP formulations.
For brevity, we present only one variant and formulation per problem, even though our techniques apply to the others as well.

\subsection{Shortest path}
\label{sec:spp}

We begin with the SPP in GCS.
The MICP we derive here is the same as that of~\cite{marcucci2024shortest} but is obtained in a simpler and more direct manner.

In the discrete SPP, we consider a directed weighted graph $G = (\cV, \cE)$ with nonnegative vertex and edge weights.
We seek a minimum-weight path (i.e., a sequence of distinct vertices connected by edges) that starts at a source vertex $\sigma \in \cV$ and ends at a target vertex $\tau \in \cV$.
The SPP can be solved in polynomial time using, e.g., Dijkstra's algorithm~\cite{dijkstra1959note}.
It is a special case of problem~\eqref{eq:graph_problem}, where the set $\cH$ of admissible subgraphs is replaced with the set $\cH\pat$ of all paths from $\sigma$ to $\tau$ in $G$.

The SPP can be formulated as an ILP of the form~\eqref{eq:graph_ilp} by substituting $\cY$ with the polytope $\cY\pat$ defined by the following conditions:
\begin{subequations}
\label{eq:path_poly}
\begin{align}
\label{eq:path_poly_ye}
& y_e \geq 0, && e \in \cE, \\
\label{eq:path_poly_yv}
& y_v \leq 1, && v \in \cV \setminus \{\sigma,\tau\}, \\
\label{eq:path_poly_yst}
& y_\sigma = y_\tau = 1, \\
\label{eq:path_poly_term}
& y_v = \sum_{e \in \incoming{v}} y_e, && v \in \cV \setminus \{\sigma\}, \\
\label{eq:path_poly_start}
& y_v = \sum_{e \in \outgoing{v}} y_e, && v \in \cV \setminus \{\tau\}.
\end{align}
\end{subequations}
Here the sets $\incoming{v}$ and $\outgoing{v}$ contain the edges that are incoming and outgoing a vertex $v \in \cV$, and, without loss of generality, we assume that $\incoming{\sigma} = \outgoing{\tau} = \emptyset$.
The first two conditions enforce bounds on the binary variables.
The third ensures that the selected path contains the source $\sigma$ and the target $\tau$.
The last two require that every vertex in the path has exactly one incoming and one outgoing edge, except for the source and target.

The polytope $\cY\pat$ does not meet the incidence condition~\eqref{eq:incidence}, since its binary elements represent a path together with disjoint cycles.
However, it meets the relaxed incidence condition~\eqref{eq:incidence_dominant}, which is sufficient for the validity of the ILP~\eqref{eq:graph_ilp} when the weights are nonnegative.
In addition, $\cY\pat$ also satisfies condition~\eqref{eq:perfect_formulation_dominant} with equality, yielding a perfect ILP formulation and allowing the SPP with nonnegative weights to be solved through a linear program (see, e.g.,~\cite[Section~13.1]{schrijver2003combinatorial}).

The SPP in GCS is obtained by replacing $\cH$ with $\cH\pat$ in problem~\eqref{eq:gcs_problem}.
The nonnegativity of the weights is translated into the nonnegativity of the objective functions $f_v$ for $v \in \cV$ and $f_e$ for $e \in \cE$.
Although the SPP is solvable in polynomial time, the SPP in GCS is NP-hard~\cite[Section~3]{marcucci2024shortest}.

The base MICP~\eqref{eq:micp} is adapted to the SPP in GCS by following the steps in Algorithm~\ref{alg:constraint_generation}.
The inequality~\eqref{eq:path_poly_ye} is mapped by our constraint-generation procedure to the valid convex constraint~\eqref{eq:micp_e0}. 
While the equalities~\eqref{eq:path_poly_term} and~\eqref{eq:path_poly_start} give us the valid linear equalities
\begin{subequations}
\label{eq:path_poly_spatial}
\begin{align}
& \bz_v = \sum_{e \in \incoming{v}} \bz_v^e, && v \in \cV \setminus \{\sigma\}, \\
& \bz_v = \sum_{e \in \outgoing{v}} \bz_v^e, && v \in \cV \setminus \{\tau\}.
\end{align}
\end{subequations}
The resulting MICP formulation of the SPP in GCS is then
\begin{subequations}
\label{eq:micp_spp}
\begin{align}
\minimize \quad & \text{objective~\eqref{eq:micp_obj}} \\
\subjectto \quad
& \by \in \cY\pat\bin, \\
& \text{constraints~\eqref{eq:micp_e0},~\eqref{eq:micp_e_hom},~\eqref{eq:path_poly_spatial}}.
\end{align}
\end{subequations}
Constraint~\eqref{eq:micp_e1} is omitted since, for each $v \in \cV$, the constraints in~\eqref{eq:path_poly} involving only the variables $y_v$ and $y_e$ for $e \in \incident{v}$ already imply the subgraph inequalities $y_v \geq y_e$ for $e \in \delta_v$.

\subsection{Traveling salesman}
\label{sec:tsp}

The TSP is one of the most famous NP-complete problems~\cite{karp1972reducibility}.
Given an undirected weighted graph $G = (\cV, \cE)$, the goal is to find a minimum-weight tour, i.e., a cycle that visits every vertex.
The TSP is a special case of the graph optimization problem~\eqref{eq:graph_problem} where the set $\cH$ is replaced with the set $\cH\tour$ of all tours in $G$.

A classical ILP formulation of the TSP is due to Dantzig, Fulkerson, and Johnson~\cite{dantzig1954solution}.
It is obtained by substituting the set $\cY$ in~\eqref{eq:graph_ilp} with the polytope $\cY\tour$ defined by the conditions
\begin{subequations}
\label{eq:tsp}
\begin{align}
\label{eq:tsp_e}
& 1 \geq y_e \geq 0, && e \in \cE, \\
\label{eq:tsp_v}
& y_v = 1, && v \in \cV, \\
\label{eq:tsp_conservation}
& \sum_{e \in \incident{v}} y_e = 2, && v \in \cV, \\
\label{eq:tsp_subtour}
& \sum_{e \in \cE_\cU} y_e \leq |\cU| - 1, && \cU \subset \cV, \ |\cU| \geq 3.
\end{align}
\end{subequations}
The first condition enforces usual bounds, the second forces the tour to visit every vertex, and the third requires that each vertex has two incident edges.
The last condition is a \emph{subtour-elimination constraint}.
The set $\cE_\cU$ on the left-hand side consists of all edges that have both ends in $\cU$.
If $\cU$ is the vertex set of a subtour (i.e., a cycle not covering every vertex), then the left-hand side equals $|\cU|$ and the constraint is violated.
These subtour-elimination constraints are exponential in number and typically enforced as \emph{lazy constraints}: they are not included in the ILP from the beginning but are added only if a candidate solution violates them during branch and bound.
Despite its exponential size, the ILP  formulation~\eqref{eq:tsp} is not be perfect (see~\cite[Section~21.4]{korte2018combinatorial} for a simple counterexample).

The TSP in GCS corresponds to problem~\eqref{eq:gcs_problem} with $\cH=\cH\tour$.
It is NP-hard since it generalizes the ordinary TSP.
Following Algorithm~\ref{alg:constraint_generation}, constraint~\eqref{eq:tsp_e} is linearized as $y_v \geq y_e \geq 0$, and yields the valid constraints~\eqref{eq:micp_e0} and~\eqref{eq:micp_e1}.
Constraint~\eqref{eq:tsp_conservation} is also linearized as $\sum_{e \in \incident{v}} y_e = 2 y_v$, and gives us
\begin{align}
\label{eq:tsp_spatial}
& \sum_{e \in \incident{v}} \bz_v^e = 2 \bz_v, && v \in \cV.
\end{align}
Our MICP formulation of the TSP in GCS is then
\begin{subequations}
\label{eq:micp_tsp}
\begin{align}
\minimize \quad & \text{objective~\eqref{eq:micp_obj}} \\
\subjectto \quad
& \by \in \cY\tour\bin, \\
& \text{constraints~\eqref{eq:micp_e0},~\eqref{eq:micp_e1},~\eqref{eq:micp_e_hom},~\eqref{eq:tsp_spatial}}.
\end{align}
\end{subequations}

\subsection{Minimum spanning arborescence}
\label{sec:msab}

As a third example, we consider the Minimum Spanning Arborescence Problem (MSAP) in GCS, which is the directed version of the MSTP in GCS illustrated in Figure~\ref{fig:demo_mstp}.
(The analysis of the MSTP is similar to the one of the TSP.)
In the ordinary MSAP, we are given a directed graph $G = (\cV, \cE)$ with a root vertex $r \in \cV$, and we seek a minimum-weight spanning arborescence in $G$, i.e., an acyclic subgraph where every vertex can be reached from the root through a unique path.
This problem is solvable in polynomial time~\cite{chu1965shortest,edmonds1967optimum}, and is a special case of the graph optimization problem~\eqref{eq:graph_problem} where $\cH=\cH\arb$ is the set of all spanning arborescences in $G$.

The polytope $\cY\arb$, defined by the following conditions, allows us to formulate the MSAP as an ILP of the form~\eqref{eq:graph_ilp}:
\begin{subequations}
\label{eq:msap}
\begin{align}
\label{eq:msap_e}
& y_e \geq 0, && e \in \cE, \\
\label{eq:msap_v}
& y_v = 1, && v \in \cV, \\
\label{eq:msap_degree}
& \sum_{e \in \incoming{v}} y_e = 1, && v \in \cV \setminus \{r\}, \\
\label{eq:msap_cutset}
& \sum_{e \in \incoming{\cU}} y_e \geq 1, && \cU \subseteq \cV \setminus \{r\}, \ |\cU| \geq 2.
\end{align}
\end{subequations}
Without loss of generality, this formulation assumes that the root has no incoming edges: $\incoming{r} = \emptyset$.
The first condition is as usual, the second ensures that every vertex is reached, and the third requires that every vertex has one incoming edge, except for the root.
The last condition is a \emph{cutset constraint} that guarantees connectivity by requiring that every subset $\cU$ of vertices, that does not include the root, has at least one incoming edge.
(We denote by $\incoming{\cU}$ the set of edges $(v,w)$ with $w \in \cU$.)
Like the subtour-elimination constraints, the cutset constraints are exponential in number and enforced as lazy constraints in practice.
Contrarily to the TSP, this exponential-size formulation of the MSAP is perfect (see, e.g.,~\cite[Corollary~52.3b]{schrijver2003combinatorial}).

The MSAP in GCS is obtained from problem~\eqref{eq:gcs_problem} by letting $\cH=\cH\arb$.
It is NP-hard as it generalizes the NP-hard MSTP with neighborhoods analyzed in~\cite[Theorem~1]{yang2007minimum}.
As already noted, our constraint-generation procedure maps constraint~\eqref{eq:msap_e} to~\eqref{eq:micp_e0}.
While the affine constraint~\eqref{eq:msap_degree} is mapped to the linear constraint
\begin{align}
\label{eq:msap_spatial}
& \sum_{e \in \incoming{v}} \bz_v^e = \bz_v, && v \in \cV \setminus \{r\}.
\end{align}
Our MICP formulation of the MSAP in GCS is then
\begin{subequations}
\label{eq:micp_msap}
\begin{align}
\minimize \quad & \text{objective~\eqref{eq:micp_obj}} \\
\subjectto \quad
& \by \in \cY\arb\bin, \\
& \text{constraints~\eqref{eq:micp_e0},~\eqref{eq:micp_e_hom},~\eqref{eq:msap_spatial}}, \\
\label{eq:micp_msap_outgoing}
& (\bz_v - \bz_v^e, 1 - y_e) \in \tilde \cX_v, && e = (v,w) \in \cE.
\end{align}
\end{subequations}
The last constraint of this MICP is part of~\eqref{eq:micp_e1}.
It is kept by Algorithm~\ref{alg:constraint_generation} since the constraints in~\eqref{eq:msap} associated with a vertex $v \in \cV$ do not imply the subgraph inequality $y_v \geq y_e$ for $e \in \outgoing{v}$.

\subsection{Facility location}
\label{sec:flp}

In the Facility-Location Problem (FLP) we are given a weighted graph $G$ that is undirected and bipartite.
The vertices $\cV$ are composed by two subsets, the facilities $\cB$ and the clients $\cC$.
Every edge connects a facility and a client.
We seek a minimum-weight assignment of each client to a facility, i.e., a subgraph where each client in $\cC$ has exactly one incident edge.
The FLP is NP-hard~\cite[Proposition~22.1]{korte2018combinatorial}, and is a special case of problem~\eqref{eq:graph_problem} where the set $\cH=\cH\assig$ consists of all assignments in $G$.

The ILP~\eqref{eq:graph_ilp} models the FLP when the polytope $\cY = \cY\assig$ enforces
\begin{subequations}
\label{eq:flp}
\begin{align}
\label{eq:flp_ve}
& 1 \geq y_v \geq y_e \geq 0, && v \in \cB, \ e \in \incident{v}, \\
\label{eq:flp_v}
& y_v = \sum_{e \in \incident{v}} y_e = 1, && v \in \cC.
\end{align}
\end{subequations}
The first condition is a standard subgraph constraint and the second ensures that each client is assigned to exactly one facility.
This ILP formulation is compact but can be weak (as expected, given the problem hardness).

The FLP in GCS corresponds to problem~\eqref{eq:gcs_problem} with $\cH=\cH\assig$.
It is NP-hard since the FLP is NP-hard.
As seen already, constraint~\eqref{eq:flp_ve} yields
\begin{subequations}
\label{eq:flp_spatial_2}
\begin{align}
& (\bz_v^e, y_e) \in \tilde \cX_v, && v \in \cV, \ e \in \incident{v}, \\
& (\bz_v - \bz_v^e, y_v - y_e) \in \tilde \cX_v, && v \in \cB, \ e \in \incident{v},
\end{align}
\end{subequations}
while constraint~\eqref{eq:flp_v} gives us
\begin{align}
\label{eq:flp_spatial_1}
& \bz_v = \sum_{e \in \incident{v}} \bz_v^e, && v \in \cC.
\end{align}
The MICP resulting from Algorithm~\ref{alg:constraint_generation} is then
\begin{subequations}
\label{eq:micp_flp}
\begin{align}
\minimize \quad & \text{objective~\eqref{eq:micp_obj}} \\
\subjectto \quad
& \by \in \cY\assig\bin, \\
& \text{constraints~\eqref{eq:micp_e_hom},~\eqref{eq:flp_spatial_2},~\eqref{eq:flp_spatial_1}}.
\end{align}
\end{subequations}

\subsection{Bipartite matching}
\label{sec:bmp}

The Bipartite-Matching Problem (BMP) is a special case of the FLP in which the number of facilities $\cB$ and clients $\cC$ is equal, and each facility must be matched with exactly one client.
The ILP formulation of the BMP is given by the following two constraints, which define the polytope $\cY\match$:
\begin{subequations}
\label{eq:bmp}
\begin{align}
& y_e \geq 0, && e \in \cE, \\
\label{eq:bmp_v}
& y_v = \sum_{e \in \incident{v}} y_e = 1, && v \in \cV.
\end{align}
\end{subequations}
This formulation is perfect (see, e.g.,~\cite[Theorem~18.1]{schrijver2003combinatorial}) and the BMP is solvable in polynomial time.

The BMP in GCS is also efficiently solvable, as it reduces to the ordinary BMP.
To show this, we fix a matching in the GCS and consider one of its edges $e=\{v,w\} \in \cE$.
The optimal values of $\bx_v$ and $\bx_w$ are independent of the other continuous variables, and can be computed via the convex program
\begin{subequations}
\label{eq:bmp_wegths}
\begin{align}
\minimize \quad &  \bc_v^\top \bx_v + \bc_w^\top \bx_w +
\bc_e^\top (\bx_v, \bx_w) \\
\subjectto \quad
& \bx_v \in \cX_v, \
\bx_w \in \cX_w, \
(\bx_v, \bx_w) \in \cX_e,
\end{align}
\end{subequations}
where the objective is linear by Assumption~\ref{ass:linear}.
We then consider the weighted graph obtained by setting the weight of each edge $e$ in the GCS to the optimal value of problem~\eqref{eq:bmp_wegths}, and all vertex weights to zero.
A matching in this weighted graph is optimal if and only if it is optimal in the original GCS.

Since the BMP in GCS is reducible to the ordinary BMP via $|\cE|$ small convex programs, mixed-integer programming is not a practical way of solving this problem.
Nevertheless, it is noteworthy that our MICP formulation is perfect for this problem, and allows us to solve the BMP in GCS via a single convex program.
Through the usual steps, our MICP is given by
\begin{subequations}
\label{eq:micp_bmp}
\begin{align}
\minimize \quad & \text{objective~\eqref{eq:micp_obj}} \\
\label{eq:micp_bmp_y}
\subjectto \quad
& \by \in \cY\match\bin, \\
& \text{constraints~\eqref{eq:micp_e0},~\eqref{eq:micp_e_hom}}, \\
\label{eq:bmp_spatial}
& \bz_v = \sum_{e \in \incident{v}} \bz_v^e, && v \in \cV,
\end{align}
\end{subequations}
where the last constraint is generated from~\eqref{eq:bmp_v}.

\begin{proposition}[Convex formulation of BMP in GCS]
\label{prop:bmp}
The convex relaxation of the MICP~\eqref{eq:micp_bmp}, obtained by replacing $\cY\match\bin$ with $\cY\match$, is exact.
\end{proposition}

\begin{proof}
See Appendix~\ref{sec:proof_bmp}.
\end{proof}

\section{Unbounded constraint sets and nonlinear objective functions}
\label{sec:nonlinear}

The techniques in the previous sections rely on Assumptions~\ref{ass:bounded} and~\ref{ass:linear}, which require the GCS constraint sets and objective functions to be bounded and linear, respectively.
Here we show that the same techniques continue to apply under the weaker assumption below.

\begin{definition}[Recession]
A \emph{recession direction} of a closed convex set $\cX \subseteq \reals^n$ is a direction $\bx \in \reals^n$ in which the set extends indefinitely, i.e., $\bx_0 + y \bx \in \cX$ for all $\bx_0 \in \cX$ and $y \geq 0$.
The \emph{recession cone} $\cX^\infty$ of $\cX$ is the set of all recession directions of $\cX$.
\end{definition}

\begin{assumption}[Superlinear vertex objectives]
\label{ass:recession_v}
For all vertices $v \in \cV$, the objective function $f_v$ grows superlinearly along all nonzero recession directions of the set $\cX_v$, i.e., $\lim_{y \rightarrow \infty} f_v(y \bx) / y = \infty$ for all $\bx \in \cX^\infty \setminus \{\bzero\}$.
\end{assumption}
Note that this assumption is trivially satisfied if the sets $\cX_v$ are bounded, since the only recession direction of a bounded set is zero.
We also highlight that the analysis of the SPP in GCS in~\cite{marcucci2024shortest} relies on this stronger boundedness assumption, which is relaxed here.

\subsection{Homogenization of unbounded sets}
\label{sec:homog_unbd}

Dropping Assumptions~\ref{ass:bounded} and~\ref{ass:linear} requires extending the homogenization operation to unbounded sets.

\begin{definition}[Homogenization of closed set]
\label{def:homog_unbd}
The \emph{homogenization} of a closed convex set $\cX \in \reals^n$ is the set
$$
\tilde \cX = \cl(\{(\bx, y) \in \reals^{n+1}:  y \geq 0, \ \bx \in y \cX \}),
$$
where $\cl$ denotes the set closure.
\end{definition}
The only difference between this definition and Definition~\ref{def:homog} is the closure, whose need is illustrated by the following example.

\begin{example}
\label{ex:hom_unbd}
Consider the unbounded interval $\cX = [1,\infty) \subset \reals$.
Its homogenization is the closure of the set $\{(x,y) : 0 < y \leq x \} \cup \{(0,0)\}$.
Taking the closure includes the positive $x$-axis (i.e., the recession cone $\cX^\infty$), yielding the set $\tilde \cX = \{(x, y) : 0 \leq y \leq x \}$ shown in Figure~\ref{fig:homog_unbounded}.
\end{example}

\begin{figure}
\centering
\begin{tikzpicture}[scale=1.3]
\draw[->] (-0.1,0) -- (3.1,0) node[right] {$x$};
\draw[->] (0,-0.1) -- (0,1.6) node[above] {$y$};
\foreach \x in {0, 1, 2} {\draw (\x,0.05) -- (\x,-0.05) node[below] {\small $\x$};}
\foreach \y in {0, 1} {\draw (0.05,\y) -- (-0.05,\y) node[left] {\small $\y$};}
\fill[blue!20,opacity=0.4] (0,0) -- (3,0) -- (3,1.5) -- (1.5,1.5) -- cycle;
\draw[thick,blue] (0,0) -- (1.5,1.5);
\draw[thick,blue] (0,0) -- (3,0);
\draw[dashed]  (0,1) -- (1,1);
\draw[dashed]  (1,0) -- (1,1);
\draw[very thick,red] (1,0) -- (3,0);
\draw[very thick,red,dashed] (1,1) -- (3,1);
\draw (1.5,0) node[above] {\color{red} $\cX$};
\draw (2,1.333333333) node {\color{blue} $\tilde \cX$};
\end{tikzpicture}
\caption{Homogenization $\tilde \cX$ of the unbounded interval $\cX = [1, \infty)$.}
\label{fig:homog_unbounded}
\end{figure}
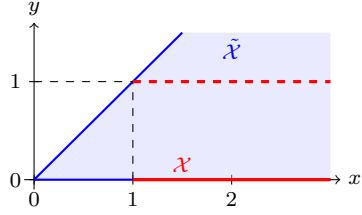

The homogenization $\tilde \cX$ of an unbounded set is still a closed convex cone, and we also still have that $(\bx, 1) \in \tilde \cX$ if and only if $\bx \in \cX$.
However, the condition $(\bx, 0) \in \tilde \cX$ is not equivalent to $\bx = \bzero$ in the unbounded case, but holds if and only if $\bx$ is a recession direction of $\cX$, i.e., $\bx \in \cX^\infty$~\cite[Theorem~8.2]{rockafellar1970convex}.

\subsection{Epigraph formulation of the GCS problem}

We handle nonlinear GCS objective functions by working with their epigraphs, so that the only difficulty is the unboundedness of the GCS constraint sets.
The latter is addressed by retracing the steps from Section~\ref{sec:method} and adapting the proofs that rely on the boundedness assumption.

We introduce slack variables $s_v$ for $v \in \cV$ and $s_e$ for $e \in \cE$, and define new GCS convex sets that are the epigraphs of the GCS objective functions:
\begin{align*}
& \cX_v' = \{(\bx_v, s_v): \bx_v \in \cX_v, \ s_v \geq f_v(\bx_v)\}, && v \in \cV, \\
& \cX_e' = \{(\bx_v, \bx_w, s_e): (\bx_v, \bx_w) \in \cX_e, \ s_e \geq f_e(\bx_v, \bx_w)\}, && e = [v,w] \in \cE.
\end{align*}
Observe that, in terms of these new sets, Assumption~\ref{ass:recession_v} states that any recession direction $(\bx_v, s_v)$ of $\cX_v'$ must be such that $\bx_v = \bzero$ and $s_v \geq 0$.
We then reformulate the GCS problem~\eqref{eq:gcs_problem} as follows:
\begin{subequations}
\label{eq:epigraph_problem}
\begin{align}
\minimize \quad & \sum_{v \in \cW} s_v + \sum_{e \in \cF} s_e \\
\subjectto \quad
& H = (\cW, \cF) \in \cH, \\
& (\bx_v, s_v) \in \cX_v', && v \in \cW, \\
& (\bx_v, \bx_w, s_e) \in \cX_e', && e = [v,w] \in \cF.
\end{align}
\end{subequations}
The only difference between this problem and the one considered in Section~\ref{sec:method} is that the new sets $\cX_v'$ and $\cX_e'$ are unbounded, and violate Assumption~\ref{ass:bounded}.

Assumption~\ref{ass:bounded} was first used to omit the constraint $\bz_v = y_v \bx_v$ from the MINCP~\eqref{eq:mincp}.
In particular, it ensured that, if $y_v=0$, then the constraint $(\bz_v, y_v) \in \tilde \cX_v$ enforces $\bz_v = \bzero$, and the objective contribution of vertex $v$ is zero.
To show that the same holds under Assumption~\ref{ass:recession_v}, consider the MINCP formulation of problem~\eqref{eq:epigraph_problem}, which parallels the MINCP~\eqref{eq:mincp}:
\begin{subequations}
\label{eq:epigraph_mincp}
\begin{align}
\minimize \quad &
\sum_{v \in \cV} t_v + \sum_{e \in \cE} t_e \\
\subjectto \quad
& \by \in \cY\bin, \\
& (\bz_v, t_v, y_v) \in \tilde \cX_v', && v \in \cV, \\
& (\bz_v^e, \bz_w^e, t_e, y_e) \in \tilde \cX_e', && e = [v,w] \in \cE, \\
& \bz_v^e = y_e \bz_v, && v \in \cV,\ e \in \incident{v},
\end{align}
\end{subequations}
where we introduced the auxiliary variables $t_v = y_v s_v$ for $v \in \cV$ and $t_e = y_e s_e$ for $e \in \cE$.
When $y_v=0$, the second constraint forces $(\bz_v, t_v)$ to be a recession direction of $\cX_v'$.
As noted above, this implies that $\bz_v = \bzero$ and $t_v \geq 0$.
Since $t_v$ does not appear in other constraints and is minimized by the objective, its optimal value is zero as desired.

The constraint-generation procedure in Lemma~\ref{lem:valid_ineq} is easily adapted to the MINCP~\eqref{eq:epigraph_mincp}.
The valid convex constraint~\eqref{eq:valid_ineq_lift} is now
\begin{align}
\label{eq:epigraph_valid_ineq_lift} 
a (\bz_v, t_v, y_v) + \sum_{e \in \incident{v}} b_e (\bz_v^e, t_v^e, y_e) \in \tilde \cX_v',
\end{align}
where $t_v^e = y_e s_v$ for all $v \in \cV$ and $e \in \incident{v}$.
This leads us to the following MICP formulation of problem~\eqref{eq:epigraph_problem}, which parallels~\eqref{eq:micp}:
\begin{subequations}
\label{eq:epigraph_micp}
\begin{align}
\minimize \quad &
\sum_{v \in \cV} t_v + \sum_{e \in \cE} t_e \\
\subjectto \quad
\label{eq:epigraph_micp_Y}
& \by \in \cY\bin, \\
\label{eq:epigraph_micp_e0}
& (\bz_v^e, t_v^e, y_e) \in \tilde \cX_v', && v \in \cV,\ e \in \incident{v}, \\
\label{eq:epigraph_micp_e1}
& (\bz_v - \bz_v^e, t_v - t_v^e, y_v - y_e) \in \tilde \cX_v', && v \in \cV,\ e \in \incident{v}, \\
\label{eq:epigraph_micp_e_hom}
& (\bz_v^e, \bz_w^e, t_e, y_e) \in \tilde \cX_e', && e = [v,w] \in \cE.
\end{align}
\end{subequations}

The second use of Assumption~\ref{ass:bounded} was in the proof of Theorem~\ref{th:micp}, where we showed that the valid convex constraints imply the original bilinear constraint.
The proof that constraints~\eqref{eq:epigraph_micp_e0} and~\eqref{eq:epigraph_micp_e1} imply $\bz_v^e = y_e \bz_v$ for all $v \in \cV$ and $e \in \incident{v}$ is essentially unchanged.
However, we must now also verify that the bilinear constraint $t_v^e = y_e t_v$ holds at optimality of the MICP.
\begin{itemize}
\item 
When $y_v=0$, we have $y_e=0$ for all $e \in \incident{v}$.
Constraint~\eqref{eq:epigraph_micp_e0} and~\eqref{eq:epigraph_micp_e1} give us $t_v^e \geq 0$ and $t_v \geq t_v^e$, respectively.
Since the variables $t_v$ and $t_v^e$ do not appear in other constraints, and $t_v$ is minimized, their optimal values are $t_v = t_v^e = 0$.
\item 
Assume that $y_v=1$.
If $y_e = 0$ for some $e \in \incident{v}$, constraints~\eqref{eq:epigraph_micp_e0} and~\eqref{eq:epigraph_micp_e1} give us 
$t_v^e \geq 0$ and $t_v - t_v^e \geq f_v (\bz_v)$.
While if $y_e = 1$, we get $t_v \geq t_v^e \geq f_v (\bz_v)$.
Therefore, at optimality, we have $t_v^e = 0$ if $y_e=0$ and $t_v^e = t_v$ if $y_e=1$.
\end{itemize}

The variants of the constraint-generation procedure in Section~\ref{sec:variants} and the multiple MICPs in Section~\ref{sec:gcs_problems} are easily adapted to the weaker assumptions of this section, analogously to how Lemma~\ref{lem:valid_ineq} is adapted in~\eqref{eq:epigraph_valid_ineq_lift}.

\section{Numerical examples}
\label{sec:examples}

We present multiple numerical examples of the GCS problems introduced in Section~\ref{sec:gcs_problems}, as well as a comparison between our MICP and alternative approaches for solving GCS problems globally.
The results can be reproduced using the \texttt{GCSOPT} library described in Section~\ref{sec:gcsopt} below.
The computer used for the experiments is a laptop with processor 2.4 GHz 8-Core Intel Core i9 and memory 64 GB 2667 MHz DDR4.
The MICP solver is \texttt{Gurobi~12.0.3}.

\subsection{Minimum-time flight of a solar-powered helicopter}
\label{sec:helicopter}

We consider a solar-powered helicopter that flies between two islands of an archipelago in minimum time.
When the battery runs low, the helicopter can take a recharging break on any island, at the cost of increasing the flight time.

The archipelago is shown in the top of Figure~\ref{fig:spp_helicopter} and composed of $I=300$ islands.
For $i=1, \ldots, I$, each island is a circle $\cC_i \subset \reals^2$ with center $\bc_i$ and radius $r_i$ drawn uniformly from the intervals $[(0,0), (5,2)]$ and $[0.02, 0.1]$, respectively.
A sampled island is rejected if it intersects with existing islands.
The start island is in the bottom left (labeled with $i=1$) and the goal island is in the top right (labeled with $i=I$).
The flight speed is constant and equal to $s=1$.
The battery level over time is described by the function $b : \reals \rightarrow [0,1]$.
It decreases at rate $\alpha = 5$ when the helicopter is flying, and increases at rate $\beta = 1$ during a recharging break.
Initially, the battery is fully charged.

\begin{figure}
\centering
\includegraphics[width=\columnwidth]{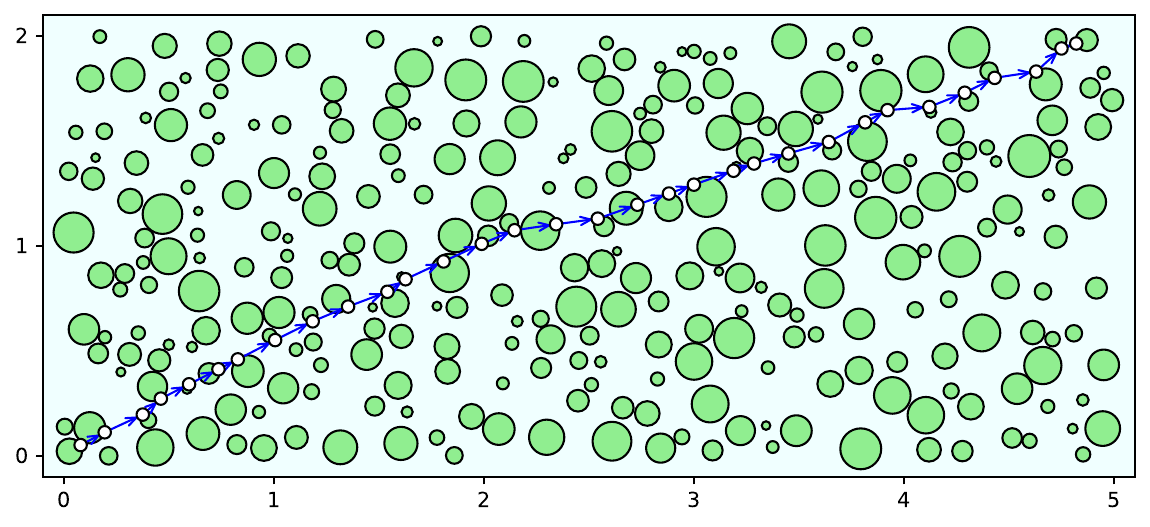}
\includegraphics[width=\columnwidth]{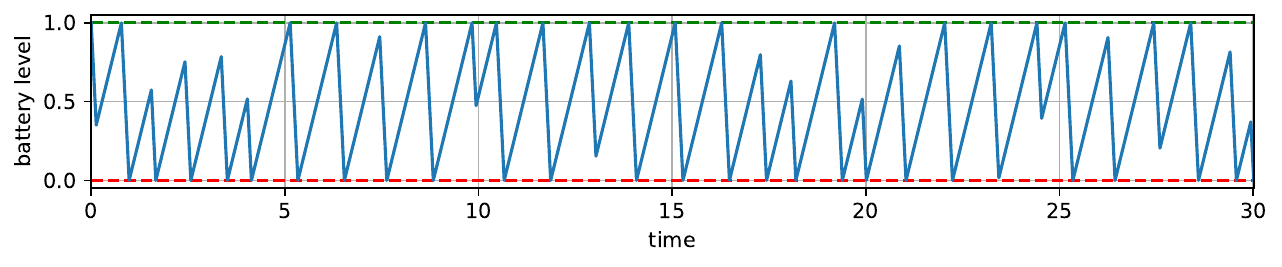}
\caption{
Example of SPP in GCS: a helicopter powered by solar energy traverses an archipelago in minimum time.
\emph{Top:} optimal flight trajectory.
\emph{Bottom:} battery level as a function of time.
}
\label{fig:spp_helicopter}
\end{figure}

Figure~\ref{fig:spp_helicopter} shows an optimal flight and the corresponding battery level $b$.
The flight requires $30$ recharging breaks and takes a total time of $30.01$.

We formulate the problem as an SPP in GCS.
Our GCS has one vertex per island, $\cV = \{1, \ldots, I\}$.
The source vertex $\sigma=1$ is paired with the start island, and the target vertex $\tau=I$ with the goal island.
Each vertex $i = 1, \ldots, I$ has continuous variable $\bx_i = (\bp_i, \bb_i) \in \reals^4$: the vector $\bp_i \in \reals^2$ represents the recharge point if the helicopter stops on the $i$th island, and $\bb_i \in \reals^2$ contains the battery level before and after the stop.
(The variables $b_{1,1}$ and $b_{I,2}$ are actually redundant, but simplify the notation.)
The recharging time on the $i$th island is a linear function of the battery levels, $t_i = (b_{i,2} - b_{i,1}) / \beta$.
The convex set $\cX_i$ forces the recharge point to lie on the island, $\bp_i \in \cC_i$, the battery levels to not exceed their limits, $\bb_i \in [0, 1]^2$, and the recharging time to be nonnegative, $t_i \geq 0$.
The set $\cX_1$ paired with the start island also ensures that the battery is fully charged at the beginning of the flight, $b_{1,2} = 1$.
The objective function of each vertex $i$ is equal to the recharging time $t_i$.

We connect two islands with an edge if the helicopter can fly between them starting with full battery.
Specifically, for $i,j =1, \ldots, I$, we have $(i,j) \in \cE$ if only if $i \neq j$ and $\|\bc_j - \bc_i\|_2 - (r_i + r_j) \leq s / \alpha$.
This leads to a graph with $|\cE| = 2146$ edges.
The cost of edge $(i,j)$ is equal to the flight time between islands $i$ and $j$, which is computed from the battery levels as $t_{ij} = (b_{i,2} - b_{j,1}) / \alpha$.
The edge constraints require that this value is not smaller than the flight distance divided by the speed: $t_{ij} \geq \|\bq_j - \bq_i\|_2 / s$.
(Although we would like to enforce this constraint as a nonconvex equality, the convex inequality is sufficient since it is always tight at optimality.)

The problem is automatically translated into the MICP~\eqref{eq:micp_spp} and solved with \texttt{Gurobi}.
The homogenization of the quadratic vertex and edge constraints leads to a Mixed-Integer Second-Order-Cone Program (MISOCP), as discussed in Section~\ref{sec:gcsopt_micp} below.
The convex relaxation of this MISOCP nearly preserves the exactness of the ordinary SPP relaxation: its optimal value is $29.96$ and the relaxation gap is only $0.2\%$.
The MISOCP solution time is $23.5$~s.\footnote{
For this problem, we set the parameter \texttt{PreMIQCPForm} to $1$, which forces \texttt{Gurobi} to handle the model as an MISOCP.
With the default setting, the convex relaxation would instead be solved through a sequence of linear programs, and the quadratic terms linearized during branch and bound.
Although this linearization can be beneficial when many iterations are needed, it is counterproductive in our case, where the convex relaxation is essentially exact.}

Figure~\ref{fig:flight_bounds} shows the progress of the lower and upper bounds during branch and bound.
The first feasible solution is found after approximately $10$~s and is globally optimal.
The rest of the time is spent proving optimality by tightening the lower bound coming from the convex relaxation.

\begin{figure}
\centering
\includegraphics[width=\columnwidth]{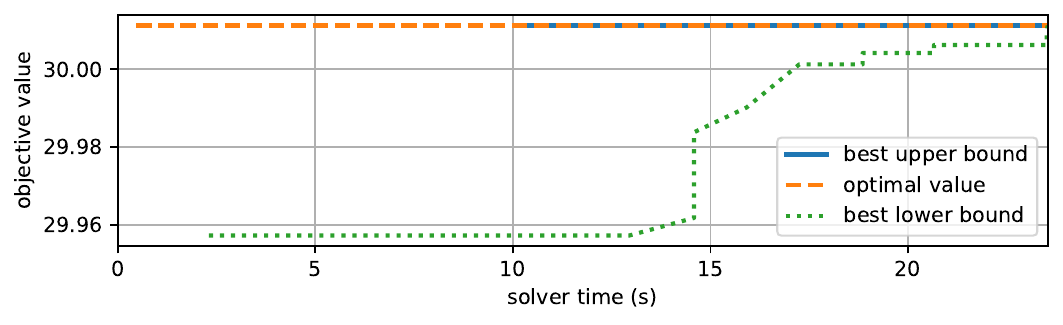}
\caption{
Bounds on the optimal value of the helicopter-flight problem during branch and bound.
The convex relaxation is almost exact (note the short range on the vertical axis) and the first feasible solution found is globally optimal.
}
\label{fig:flight_bounds}
\end{figure}

\subsection{Optimal school-bus tour}
\label{sec:bus}

It is the end of summer in Manhattan, and the local school is planning the route of the bus that will pick up the kids every morning.
Due to heavy traffic, the school decides to optimize the pick-up points and have the children walk a few blocks to meet the bus.

The school is located at the intersection of the $45$th street and the $7$th avenue: $\bp_0 := (45, 7)$.
The number of kids is $I = 18$, and the house of each kid has integer position $\bp_i$ drawn uniformly at random between $(30, 1)$ and $(59, 11)$ for $i=1, \ldots, I$.
The goal is to find integer pick-up points $\bx_i$, for $i=1, \ldots, I$, that minimize the sum of the distances traveled by the school bus and the kids (distances are measured using the $\cL_1$ norm).
The parents do not want the kids to walk for more than $d_{\max} = 3$ blocks.
Figure~\ref{fig:bus_tour} shows the optimal solution of this problem.

\begin{figure}
\centering
\includegraphics[width=\columnwidth]{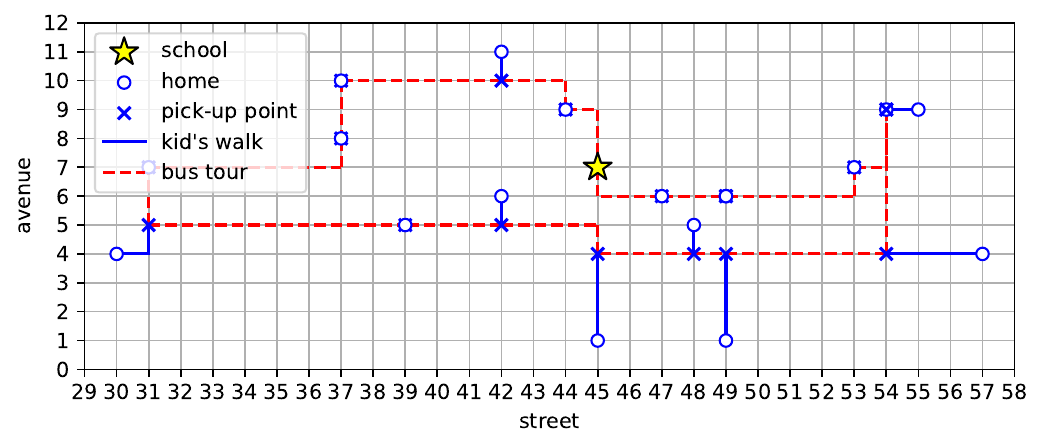}
\caption{
Optimal school-bus tour as a TSP in GCS.
The pick-up points of the $18$ kids minimize the sum of the $\cL_1$ distances traveled by the bus and the kids.
}
\label{fig:bus_tour}
\end{figure}

We solve the problem as a TSP in GCS.
We construct an undirected graph $G=(\cV, \cE)$ with $|\cV| = I + 1  = 19$ vertices.
One vertex represents the school and the others represent the kids.
The school vertex is labeled as zero and has continuous variable $\bx_0$.
The set $\cX_0$ enforces the equality $\bx_0 = \bp_0$, and the objective function $f_0$ is zero.
The kid's vertices are labeled as $i=1, \ldots, I$, and have variables $\bx_i \in \reals^2$.
The objective function for each kid is $f_i(\bx_i) = \|\bx_i - \bp_i \|_1$.
The set $\cX_i$ enforces the constraint $f_i(\bx_i) \leq d_{\max}$, as required by the parents.
The graph is fully connected, yielding a total of $|\cE| = 120$ edges.
Each edge $e = \{i,j\}$ has objective function equal to the distance traveled by the bus, $f_e(\bx_i, \bx_j) = \|\bx_j - \bx_i\|_1$.
Although this GCS allows the pick-up points $\bx_i$ to be fractional, it can be verified that there always exists an optimal solution with integral pick-up points, thanks to the $\cL_1$ metric and the integral home positions $\bp_i$.

Since the homogenization of a polyhedral set is a polyhedral cone (see Section~\ref{sec:gcsopt_micp} below), the MICP~\eqref{eq:micp_tsp} is a Mixed-Integer Linear Program (MILP) in this case.
The subtour-elimination constraints~\eqref{eq:tsp_subtour} are added as lazy constraints: every time an integer solution is found, we check if it contains subtours and, if it does, we eliminate a subtour with the smallest number of edges.
The problem has optimal value equal to $79.0$ and the solver runtime is $18.7$~s.
The MICP convex relaxation, without the subtour-elimination constraints, has optimal value $57.1$, yielding a relaxation gap of $27.7\%$.
This relatively large relaxation gap reflects the weakness of the ILP formulation~\eqref{eq:tsp} of the ordinary TSP, that our MICP builds upon.

Figure~\ref{fig:bus_bounds} shows the branch-and-bound progress.
The lower bound grows steadily towards the optimum, while the upper bound decreases very rapidly.
A near-optimal solution is found in approximately $2.5$~s, and the rest of the time is spent proving optimality.

\begin{figure}
\centering
\includegraphics[width=\columnwidth]{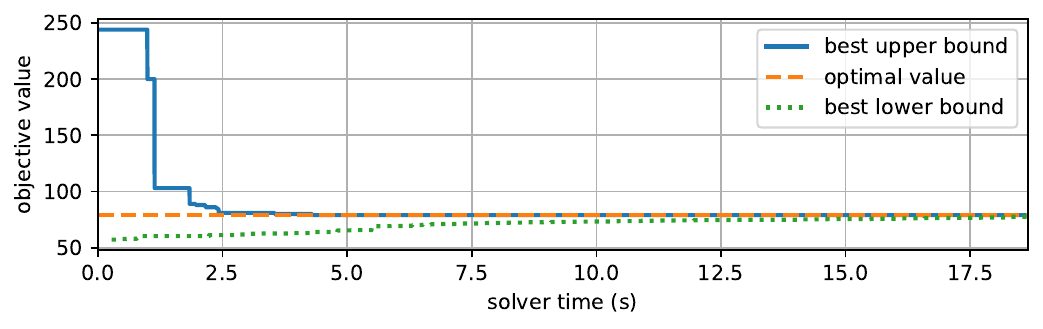}
\caption{
Bounds on the optimal value of the school-bus problem.
An optimal solution is found quickly, and most of the solve time is spent proving optimality.
}
\label{fig:bus_bounds}
\end{figure}

\subsection{Camera positioning with visibility constraints}
\label{sec:surveillance}

As an example of the MSAP in GCS, we consider the problem of positioning surveillance cameras in a floor with many rooms.
A camera is to be installed in each room, as close as possible to the room center.
Each camera must be visible to at least another camera, ensuring that any attempt to disable it is recorded by a neighboring device.
Furthermore, starting from any camera and tracing back through the camera that sees it, one must eventually reach the camera located in the main room of the building, where the central alarm system is located.
In other words, the visibility graph of the cameras must form a spanning arborescence rooted at the main room, such that every room is connected to the main room through a visibility path.
(This problem is similar in spirit to classical problems in computational geometry such as the art-gallery and watchman-route problems~\cite{o1987art,chin1986optimum}.)

The rooms are axis-aligned rectangles arranged on a uniform grid of size $I = 60$ by $J = 15$, for a total of $I J = 900$ rooms.
The room centers have coordinates $(i,j)$ for $i = 1, \ldots, I$ and $j = 1, \ldots, J$.
If $i + j$ is even (respectively, odd), the room centered at $(i,j)$ has horizontal and vertical (respectively, vertical and horizontal) sides drawn uniformly from the intervals $[4/3, 2]$ and $[2/3, 1]$.
This ensures that the room intersects only with the four neighboring rooms, centered at $(i \pm 1, j)$ and $(i, j \pm 1)$.
The main room (i.e., the room that all rooms should be connected to) has center $(0,0)$.

We model the problem via a directed GCS with one vertex per room, $|\cV| = 900$.
For each $v \in \cV$, the variable $\bx_v \in \reals^2$ represents the camera position, and the set $\cX_v$ constrains it to lie within the room boundaries.
The objective function $f_v$ is a weighted $\cL_\infty$ norm that evaluates to zero at the room center and $0.1$ at the room walls.
We connect every pair of adjacent rooms with two directed edges, yielding a total of $|\cE|=3448$ edges.
The edge objective functions are zero.
To capture visibility relations, the edge constraints must ensure that there is direct line of sight between $\bx_v$ and $\bx_w$ whenever edge $e=(v,w) \in \cE$  is selected.
Enforcing this relation exactly would require the nonconvex constraint $(1 - \lambda) \bx_v + \lambda \bx_w\in \cX_v \cup \cX_w$ for all $\lambda \in [0,1]$.
Instead, we enforce a simple sufficient condition: if edge $e$ is selected, then $\bx_w$ must lie in $\cX_v \cap \cX_w$.
This is achieved via the convex constraint set $\cX_e = \{(\bx_v, \bx_w): \bx_w \in \cX_v\}$.
The optimal solution of this problem is depicted in Figure~\ref{fig:surveillance}.

\begin{figure}
\centering
{\tiny \begin{tikzpicture}
\node[anchor=south west, inner sep=0] at (0,0) {\includegraphics[width=\columnwidth]{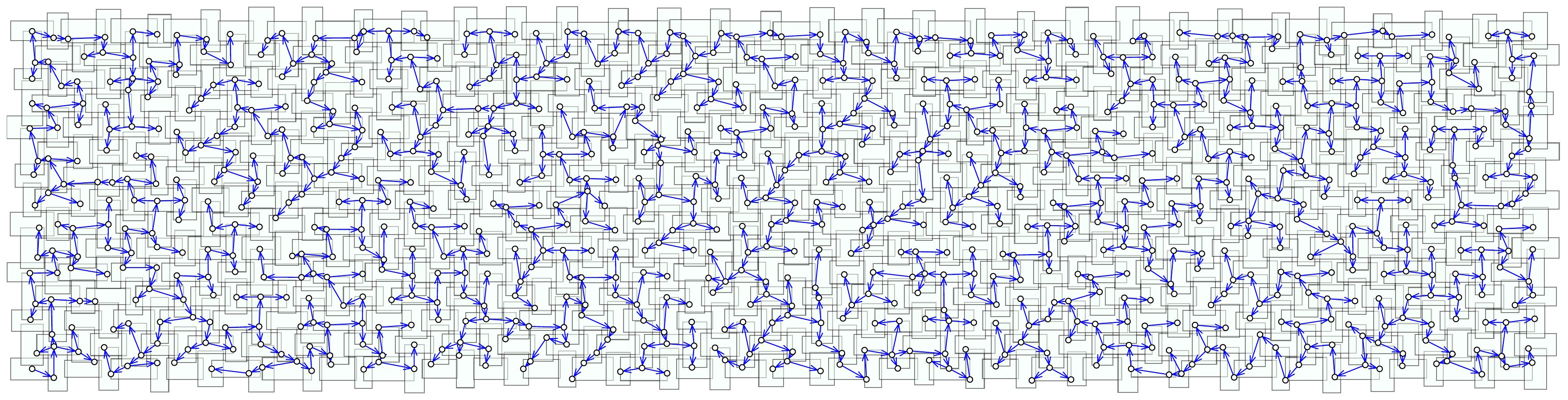}};
\def\x{0}
\def\y{0}
\def\w{1}
\def\h{1}
\def\z{2.5}
\def\dx{0}
\def\dy{3.2}
\draw[thick, red!70!black](\x,\y) rectangle ++(\w,\h);
\begin{scope}
\clip (\dx,\dy) rectangle ++(\w*\z,\h*\z);
\begin{scope}
\node[anchor=south west, inner sep=0] at (\dx-\x*\z,\dy-\y*\z) {\includegraphics[width=\z\columnwidth]{figures/surveillance.pdf}};
\end{scope}
\end{scope}
\draw[thick, red!70!black](\dx,\dy) rectangle ++(\w*\z,\h*\z);
\draw[dashed, thick, red!70!black] (\x+\w/2,\y+\h) -- (\dx+\z*\w/2,\dy);
\def\x{5.75}
\def\y{1}
\def\w{3.5}
\def\dx{\x+\w/2-\w/2*\z}
\draw[thick, red!70!black](\x,\y) rectangle ++(\w,\h);
\begin{scope}
\clip (\dx,\dy) rectangle ++(\w*\z,\h*\z);
\begin{scope}
\node[anchor=south west, inner sep=0] at (\dx-\x*\z,\dy-\y*\z) {\includegraphics[width=\z\columnwidth]{figures/surveillance.pdf}};
\end{scope}
\end{scope}
\draw[thick, red!70!black](\dx,\dy) rectangle ++(\w*\z,\h*\z);
\draw[dashed, thick, red!70!black] (\x+\w/2,\y+\h) -- (\dx+\z*\w/2,\dy);
\end{tikzpicture}}
\caption{
Optimal solution of the camera-positioning problem formulated as an MSAP in GCS.
Cyan rectangles are rooms and white dots are cameras.
Each camera is visible to at least another camera, with visibility relations represented by blue arrows.
Tracing the visibility arrows backward from any camera ultimately leads to the main room located in the bottom left.
}
\label{fig:surveillance}
\end{figure}

The MICP~\eqref{eq:micp_msap} is an MILP, with the cutset constraints~\eqref{eq:msap_cutset} enforced as lazy constraints.
For this problem, we found it more effective to eliminate all minimum-length cycles at each callback, rather than just one of them.
The problem has optimal value $25.5$ and the solver runtime is $54.6$~s.
Without the cutset constraints, the convex relaxation has optimal value $24.2$, i.e., the relaxation gap is only $5.0\%$.
This low relaxation gap is due to the strength of the original ILP formulation~\eqref{eq:msap} of the discrete MSAP.
The branch-and-bound progress is shown in Figure~\ref{fig:surveillance_bounds}: the lower bound is almost exact from the start and the solver converges immediately after finding a feasible solution.

\begin{figure}
\centering
\includegraphics[width=\columnwidth]{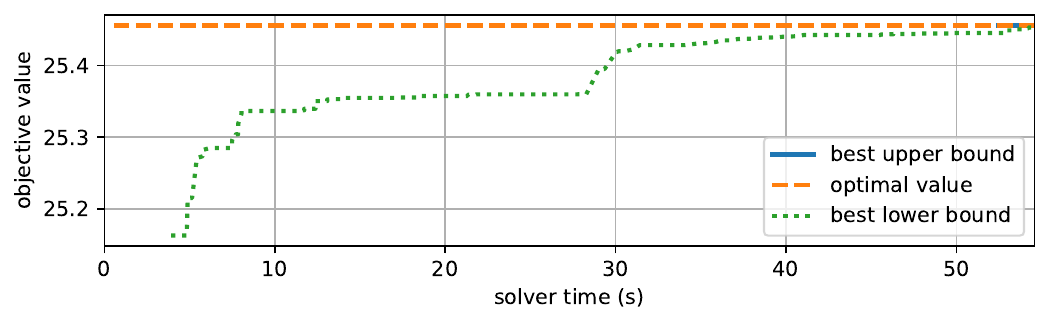}
\caption{
Bounds on the optimal value of the camera-positioning problem.
The lower bound is almost exact from the start (note the short range on the vertical axis), and the solver converges as soon as it finds a feasible solution.
}
\label{fig:surveillance_bounds}
\end{figure}

\subsection{Circle cover of two-dimensional robot link}
\label{sec:circle_cover}

Collision checking is a fundamental operation in physics engines for robotic simulation.
Given the angle of each joint, we must ensure that the robot does not collide with the environment or with itself.
Exact collision checks are often intractable, since robot links can have complex curved geometry.
A common workaround is to enclose each link with a set of simple shapes, typically spheres.
We formulate this approximation problem as an FLP in GCS.

We consider the simple two-dimensional robot link shown in Figure~\ref{fig:cover}.
This is described by a mesh with $I=17$ triangles, $\cT_i \subset \reals^2$ for $i=1,\ldots,I$.
We seek a collection of circles $\cC_j \subset \reals^2$, with $j=1,\ldots,J$, such that each triangle $\cT_i$ is contained in at least one circle $\cC_j$.
We denote by $\bc_j$ and $r_j$ the center and the radius of the circle $\cC_j$.
Among all possible solutions, we want one that minimizes the sum of the circle areas.
We impose a budget of at most $J = 5$ circles, but allow the solution to use fewer circles if advantageous.
The optimal solution is shown in Figure~\ref{fig:cover}, and uses the whole budget of five circles to cover the triangular mesh.

\begin{figure}
\centering
\includegraphics[width=.8\columnwidth]{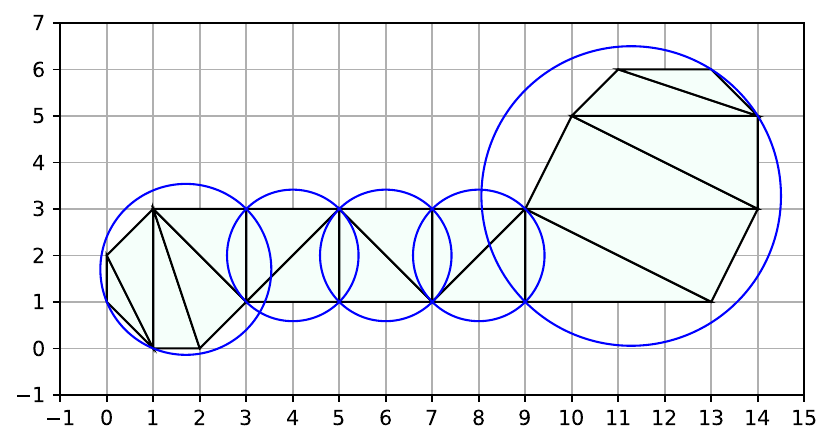}
\caption{Minimum-volume circle cover of a two-dimensional robot link represented by a triangular mesh. Problem solved as an FLP in GCS.}
\label{fig:cover}
\end{figure}

We formulate an FLP in GCS where the mesh triangles are clients, $\cC = \{1, \ldots, I\}$, and the cover circles are facilities, $\cB = \{1, \ldots, J\}$.
The total number of vertices is $|\cV| = I + J = 22$.
The clients do not need continuous variables but, to comply with our problem statement, we assign them an auxiliary continuous variable which we constrain to be zero.
Also the client objective functions are zero.
Each facility $j \in \cB$ is paired with the vector $\bx_j = (\bc_j, r_j) \in \reals^3$, and has objective function equal to the circle area $f_j(\bx_j) = \pi r_j^2$.
The set $\cX_j$ forces the center $\bc_j$ to lie in the smallest axis-aligned rectangle that contains the whole mesh, and lower bounds the radius $r_j$ with the radius of the smallest circle enclosing the smallest triangle in the mesh.
We do not enforce an upper bound on the radius $r_j$ (recall that, according to Assumption~\ref{ass:recession_v}, the superlinear growth of the objective $f_j$ with $r_j$ is sufficient for the validity of our MICP).

The GCS has $|\cE| = I J = 85$ edges, all with zero cost.
For each edge $e=\{i,j\}$, the set $\cX_e$ constrains the vertices of the triangle $\cT_i$ to lie in the circle $\cC_j$, ensuring that the solution of the FLP in GCS covers the whole mesh.

The quadratic objective and constraints make the MICP~\eqref{eq:micp_flp} an MISOCP.
The optimal value of this problem and its convex relaxation are $62.1$ and $22.2$, yielding a relaxation gap of $64\%$.
The relatively loose relaxation is this case inherited from ILP formulation~\eqref{eq:flp} of the ordinary FLP.
Despite this, the problem size is moderate, and the solver converges in $23.4$~s.

Figure~\ref{fig:cover_bounds} shows the progress of the branch-and-bound solver: the optimal solution is found in about $8$~s, but certifying optimality takes almost three times longer.
We also note that this MISOCP has $5|\cE|= 425$ quadratic constraints, which significantly burden the solver.
If we cover the mesh with squares instead of circles, we get an MISOCP with only $2|\cE|= 170$ quadratic constraints and the solver time decreases to approximately $5$~s.

\begin{figure}
\centering
\includegraphics[width=\columnwidth]{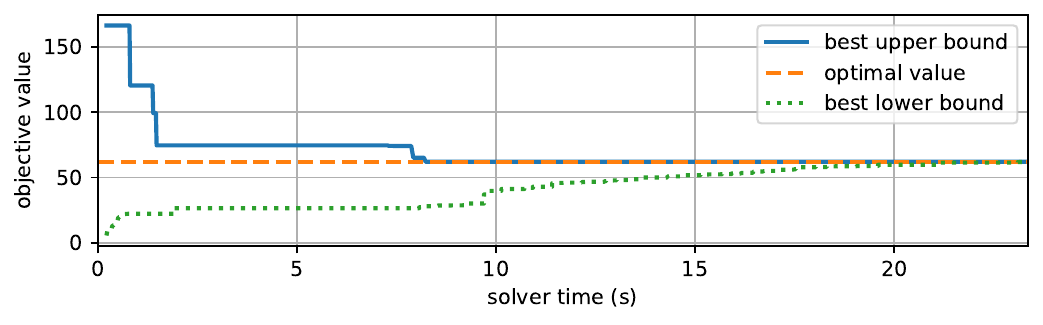}
\caption{
Bounds on the optimal value of the circle-cover problem.
An optimal cover is found in roughly $8$~s, and optimality is proven in about $23$~s.
}
\label{fig:cover_bounds}
\end{figure}

\subsection{Comparison with alternative global solution methods}
\label{sec:comparison}

We conclude this section with a comparison of our MICP formulation against two alternative approaches for solving GCS problems to global optimality:
\begin{itemize}
\item
A direct solution of the MINCP~\eqref{eq:mincp} using the spatial branch-and-bound algorithm available in \texttt{Gurobi~12}.
This algorithm has seen significant improvements in this solver release, and can handle nonconvex quadratic constraints.
This approach is representative of existing MINCP formulations for graph problems with neighborhoods (although, as noted in Section~\ref{sec:mincp}, our MINCP improves upon them in multiple directions).
\item 
A basic MICP formulation based on McCormick envelopes~\cite{mccormick1976computability} and solved with \texttt{Gurobi~12}.
For all vertices $v \in \cV$, we assume that the set $\cX_v$ is bounded and denote by $\cB_v \supseteq \cX_v$ be the smallest bounding box that contains it.
Using homogenization transformations, the standard McCormick envelope of the bilinear constraints~\eqref{eq:prod_vars} can be written as
\begin{align*}
& (\bz_v, y_v) \in \tilde \cB_v, \ (\bx_v - \bz_v, 1 - y_v) \in \tilde \cB_v, && v \in \cV, \\
& (\bz_v^e, y_e) \in \tilde \cB_v, \ (\bx_v - \bz_v^e, 1 - y_e) \in \tilde \cB_v, && v \in \cV, \ e \in \incident{v}.
\end{align*}
Replacing~\eqref{eq:gcs_bilin} with these constraints yields a correct MICP formulation of the GCS problem, which, however, is easily seen to be weaker than~\eqref{eq:micp} and our tailored MICPs.
\end{itemize}

We consider three benchmark problems, and for each problem we solve a sequence of increasingly large instances (all of which are feasible):
\begin{itemize}
\item 
The helicopter-flight problem from Section~\ref{sec:helicopter} (SPP in GCS), where we let the number of islands vary from $30$ to $300$ at increments of $30$.
In parallel, we increment the horizontal size of the archipelago from $0.5$ to $5$ at increments of $0.5$.
\item 
The school-bus problem from Section~\ref{sec:bus} (TSP in GCS), where we let the number of kids vary from $2$ to $18$ at increments of $2$.
\item 
The camera-positioning problem from Section~\ref{sec:surveillance} (MSAP in GCS), where we vary the number of rooms in each row of the grid from $I_{\min}=5$ to $I_{\max}=60$ at increments of $5$, and we keep the number of rooms in each column equal to $J = 15$.
Hence, the total number of rooms varies from $I_{\min} J = 75$ to $I_{\max}J = 900$.
\end{itemize}

Figure~\ref{fig:comparison} reports the runtimes of the three solution methods for each benchmark problem.
A missing marker indicates that the solver reached a time limit of $10^3$~s.
(We report that all the solves that reached this limit did not result in a feasible solution.)
The proposed MICP is consistently faster than the baselines, with the performance gap widening as the problem size increases.
In the helicopter-flight problem, our MICP solves all instances in at most a hundred seconds.
The MINCP approach can solve only the four smallest instances within the time limit, and the McCormick formulation only the smallest one.
A large performance gap appears also in the school-bus problem.
Our MICP solves all instances within a few tens of seconds, while, on the largest instance, the MINCP and McCormick approaches do not find a feasible solution within the time limit.
All methods perform well on the camera-positioning problem, but our MICP is still about five times faster across all instances.

\begin{figure}
\centering
\includegraphics[width=\columnwidth]{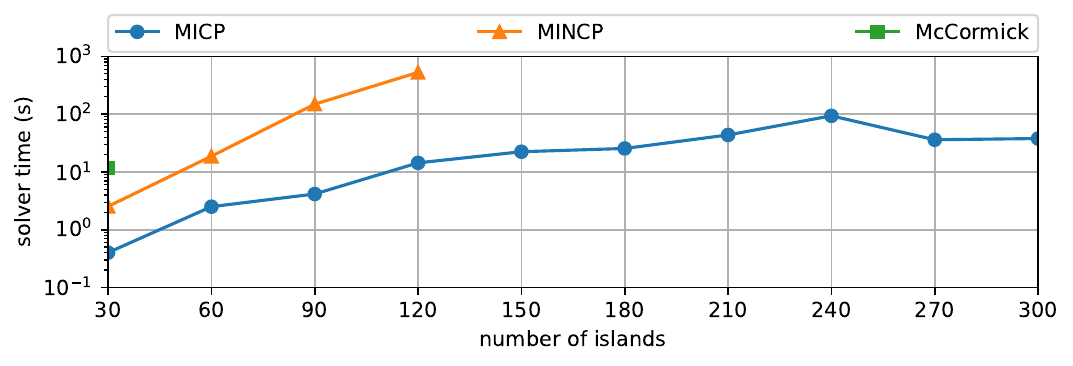} \\
\includegraphics[width=\columnwidth]{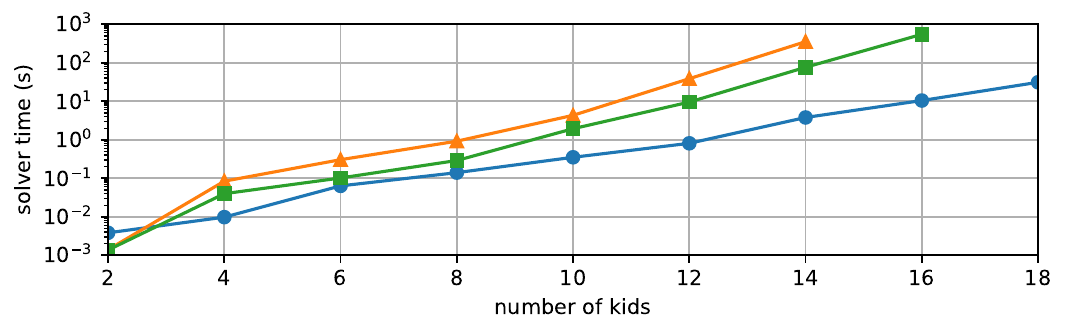} \\
\includegraphics[width=\columnwidth]{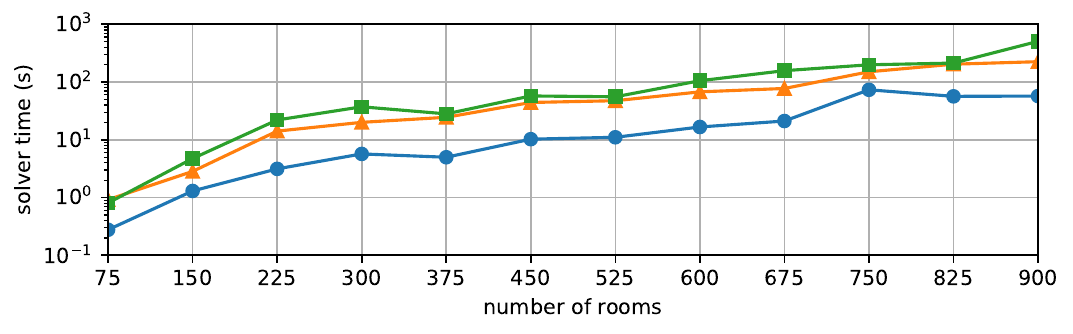}
\caption{Comparison between our MICP, a direct solution of the MINCP~\eqref{eq:micp} using spatial branch and bound, and an MICP based on McCormick envelopes.
The three approaches are evaluated on GCS problems of increasing size.
\emph{Top:} helicopter-flight problem from Section~\ref{sec:helicopter} (SPP in GCS).
\emph{Center:} school-bus problem from Section~\ref{sec:bus} (TSP in GCS).
\emph{Bottom:} camera-positioning problem from Section~\ref{sec:surveillance} (MSAP in GCS).}
\label{fig:comparison}
\end{figure}

\section{The \texttt{GCSOPT} Python library}
\label{sec:gcsopt}

This section presents \texttt{GCSOPT}: a Python library for formulating and solving GCS problems.
We illustrate the high-level goals of \texttt{GCSOPT}, its usage, and the operations it performs behind the scenes.
The library is freely available at
$$
\text{\url{https://github.com/TobiaMarcucci/gcsopt}},
$$
and can be installed via \texttt{PyPI} by running the following command in a terminal.
\begin{lstlisting}[style=bash]
pip install gcsopt
\end{lstlisting}

\subsection{High-level goals}

The primary goal of \texttt{GCSOPT} is to provide a simple framework for modeling and solving GCS problems, with emphasis on ease of use and fast prototyping.
The library offers a high-level interface that abstracts away most of the technical details discussed in this paper.
Users can define a GCS through a simple graph modeling interface combined with the syntax of \texttt{CVXPY}~\cite{diamond2016cvxpy}, a widely used Python library for convex optimization. 
All low-level operations (such as constructing the MICP, sending it to a solver, and retrieving the solution) are handled automatically behind the scenes.
By building on \texttt{CVXPY}, our library supports the definition of virtually any convex set and function, and is compatible with state-of-the-art solvers for mixed-integer optimization (such as \texttt{Gurobi}, \texttt{Mosek}, and \texttt{CPLEX}).

We emphasize that, although \texttt{GCSOPT} is open source, most mixed-integer solvers are not, and need to be installed separately.
For instance, the solver \texttt{Gurobi}, used in the experiments in Section~\ref{sec:examples}, is freely available only for academic use.
Nonetheless, competitive open-source mixed-integer conic solvers have begun to emerge (e.g., \texttt{Pajarito}~\cite{coey2020outer}), and these will be automatically accessible in \texttt{GCSOPT} as soon as a \texttt{CVXPY} interface becomes available.

We also mention that a fast implementation of the SPP in GCS is available in the open-source software \texttt{Drake}~\cite{tedrake2019drake}.
Compared to \texttt{Drake}, \texttt{GCSOPT} offers a lighter code base, a simpler and more efficient translation of GCS problems into MICPs, and can solve virtually any GCS problem.

\subsection{Illustrative example}

As a basic example of the usage of \texttt{GCSOPT}, we illustrate the Python code necessary to solve the SPP in GCS in Figure~\ref{fig:demo_spp}.

We start by importing the libraries \texttt{GCSOPT} and \texttt{CVXPY}.
The latter is used to define the convex sets and functions in our problem.

\begin{lstlisting}[style=python]
import gcsopt
import cvxpy
\end{lstlisting}

We initialize an empty directed GCS.

\begin{lstlisting}[style=python]
G = gcsopt.GraphOfConvexSets(directed=True)
\end{lstlisting}

We add $|\cV|=9$ vertices to the GCS, arranged on a square grid with side length $l=3$.
The vertex variables $\bx_v$ are constrained by the convex sets $\cX_v$ to lie within a circle of radius $r=0.3$ centered at $\bc_v = (i,j)$, for $i,j=0,\ldots,s-1$ (note that Python uses zero-based indexing).
In \texttt{GCSOPT}, every vertex must be assigned a name which, for example, can be used to retrieve the vertex from the graph instance.
Here, we name each vertex after its center $\bc_v$.

\begin{lstlisting}[style=python]
l = 3 # Grid side length.
r = 0.3 # Circle radius.
for i in range(l):
  for j in range(l):
    cv = (i, j) # Circle center and vertex name.
    v = G.add_vertex(cv) # Vertex with name cv.
    xv = v.add_variable(2) # Continuous variable of dimension 2.
    v.add_constraint(cvxpy.norm2(xv - cv) <= r) # Point in circle.
\end{lstlisting}

Next, we connect the vertices in the grid with directed edges.
The first two for loops below move through the grid, and retrieve each vertex $v$ using its name.
The variables associated with each vertex are stored in a list: to retrieve $\bx_v$ we select the zeroth and only element in that list.
In the third for loop, we connect vertex $v$ with the neighboring vertices $w$ on its right or above it.
Each edge $e=(v,w)$ has objective function $f_e(\bx_v, \bx_w) = \|\bx_w - \bx_v\|_2$.

\begin{lstlisting}[style=python]
for i in range(l):
  for j in range(l):
    cv = (i, j) # Name of vertex v.
    v = G.get_vertex(cv) # Retrieve vertex v from graph.
    xv = v.variables[0] # Get zeroth variable paired with v.
    neighbor_names = [(i + 1, j), (i, j + 1)] # Neighbors of v.
    for cw in neighbor_names:
      if G.has_vertex(cw): # Continue if a vertex is named cw.
        w = G.get_vertex(cw) # Retrieve vertex w from graph.
        xw = w.variables[0] # Get zeroth variable paired with w.
        e = G.add_edge(v, w) # Connect v and w.
        e.add_cost(cvxpy.norm2(xw - xv)) # Objective of edge e.
\end{lstlisting}

Now the GCS is fully specified, and we can solve the SPP through the MICP~\eqref{eq:micp_spp}.
This is formulated and solved automatically with the method \texttt{solve\_shortest\_path}.
The parameters of this method are the source vertex $\sigma$ and target vertex $\tau$: the source is in the bottom left and has center $\bc_\sigma = (0, 0)$, while the target is in the top right and has center $\bc_\tau = (l-1, l-1)$.

\begin{lstlisting}[style=python]
cs = (0, 0) # Source name.
ct = (l - 1, l - 1) # Target name.
s = G.get_vertex(cs) # Retrieve source vertex from graph.
t = G.get_vertex(ct) # Retrieve target vertex from graph.
G.solve_shortest_path(s, t)
\end{lstlisting}

After solving a GCS problem, \texttt{GCSOPT} automatically populates the graph with the problem result.
Below we print the optimal values of the problem and the variables paired with the vertices named $(0,1)$ and $(1,0)$.

\begin{lstlisting}[style=python]
print("Problem optimal value:", G.value)
print("Variable optimal values:")
vertex_names = [(0, 1), (1, 0)]
for cv in vertex_names:
  v = G.get_vertex(cv)
  xv = v.variables[0]
  print(cv, xv.value)
\end{lstlisting}

The last code snippet results in the following terminal output.
Observe that the second variable has no value since its vertex is not part of the optimal subgraph $H$ in Figure~\ref{fig:demo_spp}.

\begin{lstlisting}[style=bash]
Problem optimal value: 2.4561622478270677
Variable optimal values:
(0, 1) [0.24413563 0.82565037]
(1, 0) None
\end{lstlisting}

\texttt{GCSOPT} also provides basic plotting functions for visualizing a GCS and the solution of a problem.
These rely on the Python library \texttt{Matplotlib} and are limited to two-dimensional problems.
The following code generates Figure~\ref{fig:demo_spp}.

\begin{lstlisting}[style=python]
import matplotlib.pyplot as plt # Import library.
plt.figure() # Initialize empty figure.
G.plot_2d() # Plot GCS.
G.plot_2d_solution() # Plot optimal solution.
plt.show() # Show figure.
\end{lstlisting}

\subsection{Solving any GCS problem}
\label{sec:any_problem}

\texttt{GCSOPT} provides built-in functions for solving common GCS problems such as the SPP, TSP, MSTP, MSAP, and FLP.
Moreover, leveraging Algorithm~\ref{alg:constraint_generation}, it also allows the user to solve any nonstandard GCS problem of the form~\eqref{eq:gcs_problem} just by specifying a GCS and a polytope $\cY$.
As an example, let us show how the SPP in GCS illustrated above can be solved without using the method \texttt{solve\_shortest\_path}.

Starting from the graph \texttt{G} defined above, we initialize an empty list of affine ILP constraints that represent the polytope $\cY\pat$.
Then we add to this list a nonnegativity constraint for each edge binary variable, as in~\eqref{eq:path_poly_ye}.

\begin{lstlisting}[style=python]
ilp = [] # Initialize empty list of ILP constraints.
for e in G.edges:
  ye = e.binary_variable # Retrieve edge binary variable.
  ilp.append(ye >= 0) # Add constraint (18a).
\end{lstlisting}

Below are the other constraints defining the polytope $\cY\pat$, from~\eqref{eq:path_poly_yv} to~\eqref{eq:path_poly_start}.
The vertices \texttt{s} and \texttt{t} are the source and the target defined above.

\begin{lstlisting}[style=python]
for v in G.vertices:
  yv = v.binary_variable # Retrieve vertex binary variable.
  if v in [s, t]:
    ilp.append(yv == 1) # Constraint (18c).
  else:
    ilp.append(yv <= 1) # Constraint (18b).
  if v != s:
    # Sum binary variables of edges incoming vertex v.
    ye_inc = sum(e.binary_variable for e in G.incoming_edges(v))
    ilp.append(yv == ye_inc) # Constraint (18d).
  if v != t:
    # Sum binary variables of edges outgoing vertex v.
    ye_out = sum(e.binary_variable for e in G.outgoing_edges(v))
    ilp.append(yv == ye_out) # Constraint (18e).
\end{lstlisting}

The method \texttt{solve\_from\_ilp} in the following code snippet applies Algorithm~\ref{alg:constraint_generation} to the list \texttt{ilp} of affine constraints.
This automatically produces the MICP~\eqref{eq:micp_spp}, including all the constraints tailored to the SPP in GCS.
The problem is then solved, and the same optimal value as above is printed.

\begin{lstlisting}[style=python]
G.solve_from_ilp(ilp) # Solve MICP constructed by Algorithm 1.
print("Problem optimal value:", G.value)
\end{lstlisting}

The workflow just described allows the users of \texttt{GCSOPT} to easily solve complex GCS problems.
For instance, Figure~\ref{fig:inspection} shows the optimal solution of an inspection problem, where we seek a continuous closed curve of minimum Euclidean length that connects a set of designated rooms in a floor plan.
The designated rooms are red, while the other rooms are green.
Solid and dotted lines represent walls and open doors, respectively.
The optimal curve is in dashed blue.
The inspection problem can be formulated as a combination of an SPP and a TSP in GCS.
The SPP component allows us to compute minimum-length curves around obstacles as explained in~\cite{marcucci2023motion}.
The TSP component ensures that every designated room is visited at least once.
This mix of SPP and TSP constraints can be described in approximately $50$ lines of code and passed to the method \texttt{solve\_from\_ilp}, which produces the curve shown in Figure~\ref{fig:inspection}.
Note that this modeling effort is negligible and far less error prone than directly formulating the inspection problem as an MICP.

\begin{figure}
\centering
\includegraphics[width=.65\columnwidth]{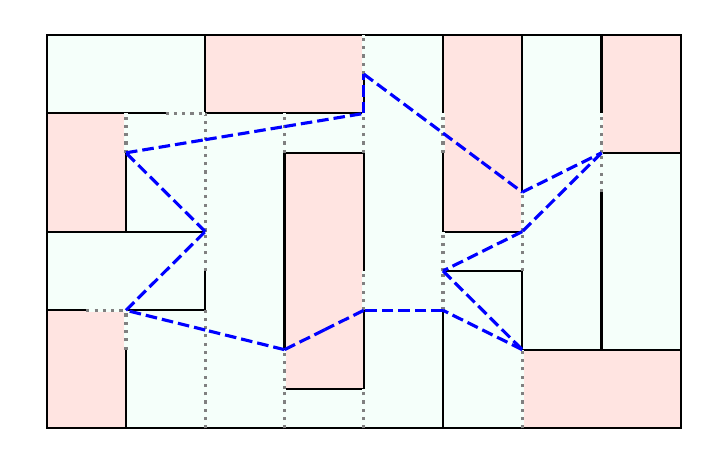}
\caption{
Inspection problem as a combination of an SPP and a TSP in GCS. 
We seek a closed curve (dashed blue) that visits a set of designated rooms (red) in a floor plan.
Solid black lines represent walls and dotted lines are doors.}
\label{fig:inspection}
\end{figure}

\subsection{Automatic formulation of the MICP}
\label{sec:gcsopt_micp}

At the core of \texttt{GCSOPT} is the observation that, as noted in~\cite[Section~2]{moehle2015perspective}, computing the homogenization of a set in conic form is particularly simple.
Recall that a closed convex set $\cX \subseteq \reals^n$ is described in \emph{conic form} if we are given a matrix $\bC \in \reals^{m \times n}$, a vector $\bd \in \reals^m$, and a closed convex cone $\cK \subseteq \reals^m$ such that
$$
\cX = \{\bx : \bC \bx + \bd \in \cK\}.
$$
For example, polyhedra, ellipsoids, and spectrahedra can all be described in conic from for appropriate choices of the cone $\cK$.
The homogenization of a closed convex set $\cX$ in conic form is simply
\begin{align}
\label{eq:homog_conic}
\tilde \cX = \{(\bx, y) : y \geq 0, \ \bC \bx + \bd y \in \cK\}.
\end{align}
(Note that here we do not take the closure as in Definition~\ref{def:homog_unbd}, since $\{\bx : \bC \bx \in \cK\}$ is exactly the recession cone $\cX^\infty$.)
Note also that this formula implies that, if we can efficiently optimize over a set $\cX$ in conic form, then we can also efficiently optimize over its homogenization $\tilde \cX$.

Using the formula~\eqref{eq:homog_conic}, the MICPs presented in this paper can be implemented in only a few tens of lines of code.
Following the steps in Section~\ref{sec:nonlinear}, we add a slack variable per vertex and edge to make all GCS objective functions linear.
Using \texttt{CVXPY} reductions and the disciplined-convex-programming rule set~\cite{grant2006disciplined}, vertex and edge constraint sets are automatically converted to conic form.
This allows us to easily enforce any constraint of the form~\eqref{eq:epigraph_valid_ineq_lift}, produced by our constraint-generation technique, as well as the edge constraint~\eqref{eq:epigraph_micp_e_hom}.

\subsection{Modeling guidelines for strong formulations}

\texttt{GCSOPT} provides a high-level interface for defining GCS problems that shields the user from the complexity of formulating an efficient MICP.
However, given the hardness of most GCS problems, and the worst-case exponential runtime of branch and bound, it is unavoidable that some of the user's modeling choices can have a noticeable effect on the solution times.
Below are simple guidelines that can help the user make the underlying MICP more efficient.

A first guideline is to keep the sets $\cX_v$ for $v \in \cV$ and $\cX_e$ for $e \in \cE$ as small as possible.
Adding vertex and edge constraints that cut unnecessary portions of these sets (but do not alter the MICP optimal value and are computationally light) can tighten the convex relaxation and accelerate the branch and bound.
For example, in the circle-cover problem in Section~\ref{sec:circle_cover}, our sets $\cX_v$ constrain each circle to be no smaller than the minimum enclosing circle of the smallest mesh triangle.
A simpler constraint would only require the radii to be nonnegative, but this would slow down the MICP solve by roughly $30\%$.

When solving nonstandard GCS problems via the method \texttt{solve\_from\_ilp}, it is important to recall that Algorithm~\ref{alg:constraint_generation} strengthens the MICP only through ILP constraints whose variables are associated with a common vertex.
Accordingly, whenever possible, it is more effective to express ILP constraints in this form rather than as ``global'' constraints involving larger groups of binary variables.
We also note that linear constraints that are redundant for the ILP can still strengthen the MICP, although they are not required for its correctness.
For example, for an MSAP with nonnegative weights, the ILP constraint~\eqref{eq:msap_degree} is redundant.
Hence, omitting it still produces a correct MICP if the GCS objective functions are nonnegative.
However, including it, together with the corresponding implied constraint~\eqref{eq:msap_spatial}, reduces the solve time of the camera-positioning problem in Section~\ref{sec:surveillance} by a factor of six.

\subsection{Edge variables}

To simplify the definition of certain constraints and objective functions, \texttt{GCSOPT} allows each edge $e \in \cE$ to be associated with auxiliary variables $\bx_e \in \reals^{n_e}$.
The edge constraint set and objective function are then $\cX_e \subseteq \reals^{n_v + n_w + n_e}$ and $f_e:\reals^{n_v + n_w + n_e} \rightarrow \reals$.
Although this can simplify the implementation, the resulting GCS problem is mathematically equivalent to the one in Section~\ref{sec:statement}.
Indeed, we can define an equivalent edge constraint set as the projection of $\cX_e$ onto the subspace of the variables $\bx_v$ and $\bx_w$, and an equivalent edge objective function as the partial minimization of $f_e$ over the extra variable $\bx_e$.
These sets and functions satisfy all convexity, closure, and boundedness assumptions required by our framework (see, e.g.,~\cite[Section~3.2.5]{boyd2004convex} for convexity).
Thus, they can replace $\cX_e$ and $f_e$, eliminating the extra variables $\bx_e$.

\section{Conclusions and future works}
\label{sec:conclusions}

This paper introduces a unified methodology for solving GCS problems, extending the ideas from~\cite{marcucci2024shortest} beyond the SPP.
Given an ILP that models an optimization problem over a weighted graph, our method automatically constructs an efficient MICP formulation for the corresponding GCS problem.
We have implemented this framework in the Python library \texttt{GCSOPT} and demonstrated its applicability through a wide range of numerical examples.

Our experiments show that the proposed MICPs often retain the strength of the ILP formulations that they build upon.
For problems such as the SPP and MSAP in GCS, the convex relaxations of our MICPs provide tight lower bounds, and the branch-and-bound solver converges in a few iterations.

As future work, we highlight that our library currently relies on general-purpose branch-and-bound solvers.
We expect that specialized optimization algorithms designed to exploit the graph structure underlying our problems could be substantially faster.
Furthermore, although already broadly applicable, the framework proposed in this paper admits several natural extensions.
It could be adapted to incorporate extended formulations~\cite{conforti2010extended} or semidefinite formulations of graph optimization problems.
It could also be extended to hypergraphs, i.e., graphs where edges can connect more than two vertices.
Beyond graphs, analogous methodologies may be developed for other classes of discrete optimization problems, such as Boolean satisfiability or equilibrium problems arising in game theory.

\appendix

\section{Proof of Proposition~\ref{prop:bmp}}
\label{sec:proof_bmp}

We use constraint~\eqref{eq:bmp_spatial} to eliminate the variables $\bz_v$ for $v \in \cV$.
After a few manipulations, the MICP convex relaxation is reduced to
\begin{subequations}
\label{eq:micp_bmp_reduced}
\begin{align}
\minimize \quad & 
\sum_{e = \{v,w\} \in \cE} (\bc_v^\top \bz_v^e + \bc_w^\top \bz_w^e + \bc_e^\top (\bz_v^e, \bz_w^e)) \\
\subjectto \quad
& \by \in \cY\match, \\
& (\bz_v^e, y_e) \in \tilde \cX_v, && v \in \cV,\ e \in \incident{v}, \\
& (\bz_v^e, \bz_w^e, y_e) \in \tilde \cX_e, && e = \{v,w\} \in \cE.
\end{align}
\end{subequations}
Given an edge $e=\{v, w\} \in \cE$, we isolate the objective term and the constraints that involve the variables $\bz_v^e$ and $\bz_w^e$.
This gives us the subproblem
\begin{align*}
\minimize \quad
& \bc_v^\top \bz_v^e + \bc_w^\top \bz_w^e + \bc_e^\top (\bz_v^e, \bz_w^e) \\
\subjectto \quad
& (\bz_v^e, y_e) \in \tilde \cX_v, \
(\bz_w^e, y_e) \in \tilde \cX_w, \
(\bz_v^e, \bz_w^e, y_e) \in \tilde \cX_e.
\end{align*}
For any fixed value of $y_e$, this subproblem can be solved independently of the rest of the convex relaxation.
By substituting $\bz_v^e = y_e \bx_v$ and $\bz_w^e = y_e \bx_w$, we see that its optimal value is equal to $y_e c_e$, where $c_e$ is the optimal value of~\eqref{eq:bmp_wegths}.
Therefore, the convex relaxation~\eqref{eq:micp_bmp_reduced} is equivalent to
\begin{align*}
\minimize \quad & 
\sum_{e \in \cE} c_e y_e \\
\subjectto \quad
& \by \in \cY\match.
\end{align*}
Since this is the convex relaxation of the ordinary BMP, which is exact as noted in Section~\ref{sec:bmp}, it follows that the convex relaxation of the MICP~\eqref{eq:micp_bmp} is also exact.

\bibliographystyle{spmpsci}
\bibliography{bibliography.bib}

\end{document}

%% file: defs.tex
\newcommand{\reals}{\mathbb R}

\newcommand{\nnreals}{\mathbb R_{\geq 0}}
\newcommand{\nnintegers}{\mathbb Z_{\geq 0}}
\newcommand{\minimize}{\mathrm{minimize}}

\newcommand{\subjectto}{\mathrm{subject \ to}}
\newcommand{\tand}{\mathrm{\ and \ }}

\newcommand{\tforall}{\mathrm{ \ for \ all \ }}

\newcommand{\conv}{\mathrm{conv}}
\newcommand{\cl}{\mathrm{cl}}

\newcommand{\cB}{\mathcal B}
\newcommand{\cC}{\mathcal C}

\newcommand{\cE}{\mathcal E}
\newcommand{\cF}{\mathcal F}

\newcommand{\cH}{\mathcal H}
\newcommand{\cI}{\mathcal I}

\newcommand{\cK}{\mathcal K}
\newcommand{\cL}{\mathcal L}

\newcommand{\cT}{\mathcal T}
\newcommand{\cX}{\mathcal X}
\newcommand{\cY}{\mathcal Y}
\newcommand{\cU}{\mathcal U}
\newcommand{\cV}{\mathcal V}
\newcommand{\cW}{\mathcal W}

\newcommand{\bzero}{\bm 0}

\newcommand{\bb}{\bm b}
\newcommand{\bc}{\bm c}
\newcommand{\bd}{\bm d}

\newcommand{\bp}{\bm p}

\newcommand{\bq}{\bm q}

\newcommand{\bx}{\bm x}
\newcommand{\by}{\bm y}
\newcommand{\bz}{\bm z}

\newcommand{\bC}{\bm C}

\newcommand{\bin}{^{\mathrm{bin}}}
\newcommand{\sub}{_{\mathrm{sub}}}
\newcommand{\pat}{_{\mathrm{path}}}
\newcommand{\tour}{_{\mathrm{tour}}}

\newcommand{\arb}{_{\mathrm{arb}}}
\newcommand{\assig}{_{\mathrm{assig}}}
\newcommand{\match}{_{\mathrm{match}}}

\SetKwInput{KwIns}{Inputs}

\newcommand{\incident}[1]{\delta_#1}
\newcommand{\incoming}[1]{\delta_#1^-}
\newcommand{\outgoing}[1]{\delta_#1^+}

\newtheorem{assumption}{Assumption}

\definecolor{codegreen}{rgb}{0,0.6,0}
\definecolor{codepurple}{rgb}{0.58,0,0.82}
\definecolor{backcolour}{rgb}{0.97,0.97,0.95}

\lstdefinestyle{python}{
backgroundcolor=\color{backcolour},   
commentstyle=\color{codegreen},
keywordstyle=\color{magenta},
numberstyle=\ttfamily\color{gray},
stringstyle=\color{codepurple},
basicstyle=\small\ttfamily,
breakatwhitespace=false,         
breaklines=true,                 
captionpos=b,                    
keepspaces=true,                 
showspaces=false,                
showstringspaces=false,
showtabs=false,                  
tabsize=2,
language=Python,
morekeywords=[1]{as,assert,nonlocal,with,yield,self,True,False,None},
}

\lstdefinestyle{bash}{
	backgroundcolor=\color{backcolour},
	basicstyle=\small\ttfamily\color{black},
	keywordstyle=\color{termcommand}\bfseries,
	commentstyle=\color{gray}\itshape,
	showstringspaces=false,
	breaklines=true,
	breakatwhitespace=false,
	tabsize=2,
	frame=single,
	rulecolor=\color{backcolour},
	morekeywords={\$}, 
}